\documentclass{amsart}
\usepackage{aliascnt} 
\usepackage{amssymb}
\usepackage{amsmath}
\usepackage{amsthm}
\usepackage{amsbsy}
\usepackage{bm}
\usepackage{hyperref}
\usepackage{cleveref}
\date{\today}
\usepackage{cite}
\usepackage{array}
\usepackage{graphicx}
\usepackage{float}
\usepackage{upgreek}
\usepackage{color}

\usepackage{enumerate}
\usepackage{enumitem}
\newtheorem{thm}{Theorem}[section]
\newtheorem{prop}[thm]{Proposition}
\newtheorem{lem}[thm]{Lemma}
\newtheorem{cor}[thm]{Corollary}
\newtheorem{exm}[thm]{Example}
\newtheorem{df}[thm]{Definition}
\newtheorem{rem}[thm]{Remark}
\newtheorem{fact}{Fact}

\newcommand {\R} {\mathbb{R}}

\def\v{\bf v}
\def\ii{{\bf i}}
\def\jj{{\bf j}}
\newtheorem{nt}[thm]{Note}
\newtheoremstyle{non-italic} 
{3pt}    
{3pt}    
{\normalfont} 
{}       
{\bfseries} 
{.}      
{ }      
{}       

\theoremstyle{non-italic}

\makeatletter
\@namedef{subjclassname@1991}{Subject}
\@namedef{subjclassname@2000}{Subject}
\@namedef{subjclassname@2010}{Subject}
\@namedef{subjclassname@2020}{Subject}
\makeatother

\date{\today}

\subjclass[2020]{Primary 
	05C65 
	; Secondary 
	05C50, 
}

\usepackage{tikz-cd}
\usepackage{pgf,tikz,pgfplots}
\usepackage{graphicx}
\usepackage{subcaption}
\tikzstyle{none}=[]
\tikzstyle{new style 0}=[draw,circle,fill=white]
\tikzstyle{new edge style 1}=[draw,dashed]
\tikzstyle{new edge style 1}=[draw,dashed]
\usetikzlibrary{shapes,fit,backgrounds}
\usetikzlibrary{calc}
\usetikzlibrary{positioning}  

\theoremstyle{theorem}
    \newtheorem{theorem}{Theorem}
    \newtheorem{lemma}{Lemma}

\theoremstyle{definition} 

    \newtheorem{remark}{Remark}
    \newtheorem{example}[theorem]{Example}
    \newtheorem{exercise}[theorem]{Exercise}
    
    \newtheorem{assumption}{Assumption}

\def\suchthat{\; : \;}

\def\C{\mathbb{C}}

\def\N{\mathbb{N}}

\def\R{\mathbb{R}}

\def\l{\left}
\def\r{\right}
\def\<{\langle}
\def\>{\rangle}

\newcommand{\E}{\mbox{\bf E}}

\def\bar{\overline}
\def\P{{\bf P}}

\def\given{\left.\vphantom{\hbox{\Large (}}\right|}

\newcommand\Tr{{\mbox{Tr}}}

\newcommand\mnote[1]{} 
\newcommand\be{\begin{equation*}}

\newcommand\ee{\end{equation*}}

\newcommand\ben{\begin{equation}}
\newcommand\een{\end{equation}}
\newcommand\bes{\begin{eqnarray*}}
\newcommand\ees{\end{eqnarray*}}

\newcommand\bex{\begin{exercise}}
\newcommand\eex{\end{exercise}}
\newcommand\beg{\begin{example}}
\newcommand\eeg{\end{example}}
\newcommand\benu{\begin{enumerate}}
\newcommand\eenu{\end{enumerate}}
\newcommand\beit{\begin{itemize}}
\newcommand\eeit{\end{itemize}}
\newcommand\berk{\begin{remark}}
\newcommand\eerk{\end{remark}}
\newcommand\bdefn{\begin{defintion}}
\newcommand\edefn{\end{definition}}
\newcommand\bthm{\begin{theorem}}
\newcommand\ethm{\end{theorem}}
\newcommand\bprf{\begin{proof}}
\newcommand\eprf{\end{proof}}
\newcommand\blem{\begin{lemma}}
\newcommand\elem{\end{lemma}}

\newcommand{\sm}{{\raise0.3ex\hbox{$\scriptstyle \setminus$}}}

\def\l{\left}
\def\r{\right}

\def\eps{\epsilon}

\def\CHI{\mathchoice%
{\raise2pt\hbox{$\chi$}}%
{\raise2pt\hbox{$\chi$}}%
{\raise1.3pt\hbox{$\scriptstyle\chi$}}%
{\raise0.8pt\hbox{$\scriptscriptstyle\chi$}}}
\def\smalloplus{\raise1pt\hbox{$\,\scriptstyle \oplus\;$}}

\title[local weak convergence of sparse random uniform hypergraphs]{spectrum and local weak convergence of sparse random uniform hypergraphs}

\author{Kartick Adhikari}
\address{Department of Mathematics, Indian Institute of Science Education and Research, Bhopal 462066}

\email{kartick [at] iiserb.ac.in}

\author{Samiron Parui}
\address{Department of Mathematics, Indian Institute of Science Education and Research, Bhopal 462066}
\email{samironparui@gmail.com}

\date{\today}

\thanks{The research of KA was partially supported by the Inspire Faculty Fellowship: DST/INSPIRE/04/2020/000579.}

\begin{document}
		\keywords{Hypergraphs; Random hypergraphs; Weak limit; Limiting spectral distribution}
		\begin{abstract}
		The notion of the local weak convergence, also known as Benjamini-Schramm 
		convergence, for a sequence of graphs, was introduced 
		by Benjamini and Schramn \cite{Benjamini-Schramm-convergence}. It is well known that the 
		local weak limit of the sparse Erd\H os-Re\'nyi graphs is the Galton-Watson measure
		with Poisson offspring distribution almost surely. Recently in 
		\cite{adhikari2023spectrumrandomsimplicialcomplexes}, Adhikari, Kumar and Saha 
		showed that the local weak limit of the line graph of the sparse Linial-Meshulam complexes is the $d$-block Galton-Watson measure almost surely. 
		In this article, we study the local weak convergence of a unified model, namely, the weighted line graphs of sparse $k$-uniform random hypergraphs on $n$ vertices. 
		
		Suppose that $ H(n,k,p)$ denotes the $k$-uniform random hypergraphs on $n$ vertices, 
		that is, each $k$-set (a subset of $[n]$ with cardinality $k$) is included as a 
		hyperedge with probability $p$ and independently. 
		Let $H=(V,E)$ be a $k$-uniform hypergraph on $n$ vertices and $1\le r\le k-1$.  The $r$-set weighted line graph of $H$ is a weighted graph $G_r(H)=(\mathcal V_r, \mathcal E_r,w_H)$, where 
		$$
		\mathcal V_r=\l[\binom{n}{r}\r], \;\; \mathcal E_r=\{\{\tau_1,\tau_2\} \suchthat \tau_1,\tau_2\in \mathcal V_r, \exists e\in E \mbox{ s.t. } \tau_1,\tau_2\subset e\},
		$$
		$w_H(\{\tau_1,\tau_2\})=|\{e \in E\suchthat \tau_1,\tau_2\subset e\}|$, and $|\cdot|$ denotes the number of elements in the set.
		In particular, if $H_n\in H(n,k,p)$ then  $G_1(H_n)$ is a natural generalization 
		of the Erd\H os-R\' enyi graphs, and $G_{k-1}(H_n)$ is the line graph of 
		the Linial-Meshulam complex. We show that if $\binom{n-r}{k-r}\to \lambda$ as $n\to \infty$ then the local weak limit of $G_{r}(H_n)$ is the $(\binom{k}{r}-1)$-block Galton-Watson measure with 
		Poisson$(\lambda)$ offspring distribution almost surely. As a consequence, we derive the limiting spectral distribution of the adjacency matrices of $G_r(H_n)$.
	\end{abstract}
	
		\maketitle

	\section{Introduction}
	In the late 1950s and early 1960s, Paul Erdős and Alfréd Rényi introduced the notion of a \emph{random graph} \cite{erdos-reyni-1959,erdos-reyni-1961,erdos-reyni-1960,erdos1960evolution}. Informally, a random graph is a graph-valued random variable, meaning that it takes values in the space of graphs, each occurring with a prescribed probability. Following their pioneering work, several important random graph models have been developed and extensively studied. These include the uniform model $G(n, M)$, in which a graph is chosen uniformly at random from all graphs with $n$ vertices and $M$ edges; the Gilbert model $G(n, p)$, where each potential edge appears independently with probability $p$; the configuration model, which generates graphs with prescribed degree sequences; the preferential attachment model, which produces graphs with power-law degree distributions as observed in real-world networks; as well as models involving random adjacency matrices, inhomogeneous random graphs, and graphons (see \cite{Gilbert-RG,Lovasz-graph-limit5,Spec-norm,exp-spec-degree} and references therein). 
	
	Various structural properties of these models have been analyzed, including threshold phenomena (where certain structures emerge abruptly as a parameter crosses a critical value), phase transitions, almost sure properties, and asymptotic behavior as the number of vertices tends to infinity, i.e., $n \to \infty$ \cite{BOLLOBAS1980311,Molloy-critical2,rand-conn,scaling-random}. A classical result in the study of asymptotic distributions is the Poisson limit theorem (also known as \emph{the law of rare events}), established by Siméon Denis Poisson in the early 19th century, which states that the binomial distribution converges to the Poisson distribution under appropriate conditions. This result forms the basis for many asymptotic theories, such as modelling rare events in large systems, the law of small numbers, and local limit theorems \cite{barbour1992poisson,arratia1989two,adhikari2023spectrumrandomsimplicialcomplexes,bordenave2012lecture}.
	
	A significant development in this direction is the notion of local convergence 
	(also known as Benjamini–Schramm convergence) for sparse random graphs 
	\cite{Aldous2004,Benjamini-Schramm-convergence,MR3453369}. Specifically, 
	when $G(n, p)$ satisfies $np = \lambda$ for a fixed $\lambda > 0$, the graph rooted at a 
	uniformly chosen vertex converges almost surely to a distribution called Galton–Watson
	measure (see \cite{harris1948branching,athreya2012branching}). This result has been extended to higher-dimensional objects such as the Linial–Meshulam random simplicial complex in \cite{adhikari2023spectrumrandomsimplicialcomplexes}. In the present work, we investigate the local weak limits of $k$-uniform random hypergraphs.
	
	A hypergraph extends the notion of a graph by allowing its edges, called \emph{hyperedges}, 
	to connect more than two vertices. This generalization makes hypergraphs particularly 
	well-suited for modeling systems with higher-order interactions, beyond pairwise relationships.
	Such structures naturally arise in various contexts, including group communications, 
	collaborative research networks, and hierarchical organizational 
	frameworks \cite{multibody-1,multibody_2020,multibody-2,bai2021hypergraph}. Formally, 
	whereas a standard graph comprises edges as two-element subsets of the vertex set, 
	a hypergraph permits hyperedges to be arbitrary subsets with cardinality greater than or 
	equal to two. 
	A \emph{$k$-uniform hypergraph} is a hypergraph in which the cardinality of 
	each hyperedge is $k$.

	The Erdős–Rényi random graph $G(n, p)$ admits a natural generalization to the $k$-uniform 
	random hypergraphs. See \cite{random-hyp1,random-hyp2,ZhuYizhe_Community_detection}, among others.
	For $k\le n$, given the 
	vertex set $[n] = {1, 2, \ldots, n}$, the hyperedges of a $k$-uniform random hypergraph are 
	subsets drawn from $\mathcal{ V}_k=\left[\binom{n}{k}\right]$, the collection of all 
	$k$-element subsets of $[n]$, independently with probability $p$. 
	
	For graphs, the local structure around a vertex is typically captured by a rooted graph 
	centered at the vertex. In the context of $k$-uniform hypergraphs, the notion of a root can be 
	extended beyond vertices to include $r$-cardinality subsets of the vertex set, 
	where $1 \le r \le k - 1$. Such connections have been explored in the literature under various
	frameworks, including $r$-intersection graphs 
	\cite{naik2019intersection-r-set,intersection-hypergraph-algo}, $r$-walks in 
	hypergraphs \cite{r-walk}, and other related notions. We take an $r$-subset of the vertex set as the root and refer to the resulting object as an $r$-set rooted $k$-uniform hypergraph.
	
	In this work, we show that the local structure of such hypergraphs converges to a 
	random infinite $r$-set rooted hypergraph, which can be regarded as a generalization of both
	the Galton–Watson measure with Poisson offspring and the limiting object is similar to a 
	$d$-block Galton–Watson measure with Poisson offspring distribution mentioned 
	in \cite{adhikari2023spectrumrandomsimplicialcomplexes}, see Section \ref{r-glt} for precise definition. In particular, for the case $r=1$ the multi-type Galton–Watson hypertree studied in \cite{ZhuYizhe_Community_detection, Zhu-com}.
	The local structure of $r$-set rooted $k$-uniform hypergraph is captured by the $r$-set 
	rooted local weak limit of $k$-uniform hypergraphs. See Section \ref{local-wk} for a 
	detailed description of $r$-set rooted local weak limit and associated topology.  
	In Theorem \ref{main}, we show that the $r$-set rooted local weak limit of sparse  $k$-uniform 
	random hypergraphs tends to the $(\binom{k}{r}-1)$-block Galton–Watson measure
	with Poisson offspring. For $k=2$ and $r=1$, the Theorem \ref{main} leads to the well-known 
	Benjamini-Schramm convergence, 
	see Corollary \ref{cor-Benjamini and Schramm}.
	For the $r=1$ case, when each vertex of the 
	random hypergraph serves as a root, the local weak limit of sparse random $k$-uniform 
	hypergraphs tends to the $(k-1)$-block Galton–Watson measure
	with Poisson offspring, see Theorem \ref{deviationto0}. For $r=k-1$, the $(k-1)$-rooted  
	$k$-uniform hypergraphs resembles the Linial-Meshulam model \cite{Linial,Meshulam}. In this case as well, the limiting measure is the $(k-1)$-block Galton–Watson measure with Poisson offspring (see Corollary \ref{Linial-Meshulam}). This result has been  
	established in \cite{adhikari2023spectrumrandomsimplicialcomplexes}.
	
	Following the foundational work on local weak convergence for sparse random graphs 
	one of the key developments has been its application to the study of the limiting spectral 
	distribution of the adjacency operator (see \cite{bordenave2010resolvent, bordenave2010rank,
		bordenave2017mean, SALEZ2015249}). In particular, this notion of convergence implies that, 
	in the limit as the number of vertices tends to infinity, the neighborhood of a typical vertex
	in a sparse Erdős–Rényi graph resembles a Galton–Watson measure with Poisson offspring distribution.
	
	Since the behavior of graph operators such as the adjacency matrix is governed by the 
	local structure of the graph, this local convergence has direct implications for spectral 
	properties. For a finite graph $G$, the  adjacency matrix $A_G$ of $G$ acts on any complex-valued function
	$f$ on its vertex set by
	$$
	(A_G f)(v) = \sum_{u \sim v} f(u),
	$$
	involving only the neighbors of $v$. The local similarity between sparse random 
	graphs and Galton–Watson trees translates into a corresponding similarity in the behavior of 
	their adjacency operators. As a result, the spectral measures associated with the adjacency
	matrices of such random graphs exhibit convergence 
	(see \cite{bordenave2017mean, bordenave2016spectrum, bordenave2010resolvent}).
	
	This type of spectral convergence has also been extended to random simplicial complexes, 
	as shown in \cite[Section~3]{adhikari2023spectrumrandomsimplicialcomplexes}. In the present work,
	we extend this framework to random hypergraphs by leveraging our results on local convergence in
	the space of $r$-set rooted $k$-uniform  random hypergraphs. In Theorem \ref{spe-conv}, we establish 
	the limiting spectral distribution of the $r$-set adjacency matrices of $k$-uniform random 
	hypergraphs. When $r=1$, this matrix generalizes several classical notions of hypergraph 
	adjacency matrices (see \cite{Banerjee-hypmat, samiron-23}), and its limiting spectral 
	distribution is analyzed in Corollary \ref{spec-r=1}. For $r = k-1$, the limiting spectral behavior
	aligns with that of the Linial–Meshulam model, as discussed in 
	\cite{adhikari2023spectrumrandomsimplicialcomplexes}, and is presented in Corollary \ref{spec-r=k-1}.
	
	The rest of the paper is organized as follows. In Section \ref{prili}, we recall all the essential definitions
	and notations to state our main results. The main results and their corollaries are stated 
	in Section \ref{se:mainresult}. The proofs of Theorem \ref{deviationto0} and Theorem \ref{main} are given in 
	Section \ref{se:r=1} and Section \ref{main-proof} respectively. The proofs of both results proceed in two stages. Convergence in expectation is shown in Section \ref{conv-exp-1-pr} and Section \ref{conv-exp-r}, while almost sure convergence is established in Proposition \ref{almost-sure-1} and Section \ref{r-almost-sure}. In Section \ref{spec-result}, we provide the proofs of 
	Theorem \ref{spe-conv} and its corollaries.
	
	
	\section{Preliminaries}\label{prili} 
	We begin by setting out the basic concepts that will be used throughout: random graphs and hypergraphs, rooted graphs and their associated topology, weak convergence, and the Galton–Watson measure. These notions form the theoretical framework for analyzing the weak convergence of sparse random hypergraphs to their limiting structures. 
	We begin by recalling the Erd\H{o}s–R\'enyi random graph \cite{erdos1960evolution}. This serves as a foundation for extending the concept to uniform hypergraphs.
	
	\subsection{Random graphs and hypergraphs} 
	Let $V\subseteq \N$ be a non-empty set and $E\subset \{\{x,y\}\suchthat x,y\in V, x\neq y\}$.
	A {\it weighted   graph} $G$ is an ordered triple $G=(V,E,w)$, where $w:E\to \R$ is a real-valued function on $E$. A multigraph $G$ is a {\it weighted graph} with positive integer valued weights, that is, $w(E)\subseteq \N$. If $w\equiv 1$ then the $G$ is called a simple graph.
	
	A \emph{hypergraph} $H$ generalizes the concept of a graph by allowing \emph{hyperedges}, which may connect more than two vertices. Formally, a hypergraph $H$ is an ordered pair $H=(V(H), E(H))$, where $V(H)$ is a non-empty set of vertices, and $E(H)$ is a collection of non-empty subsets of $V(H)$, called hyperedges. Given an integer $k \leq |V(H)|$, a hypergraph $H$ is called \emph{$k$-uniform} if every hyperedge $e \in E(H)$ has cardinality $|e| = k$. Recall that $|\cdot|$ denotes numbers of elements in the set.
	
	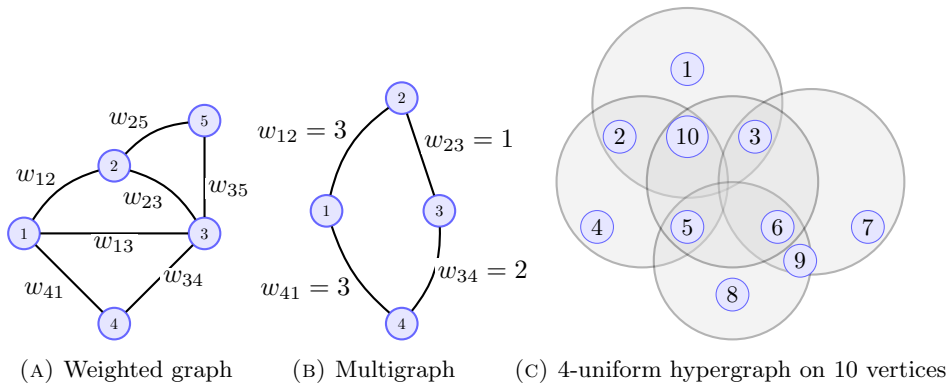
\begin{figure}[ht]
		\centering
		\begin{subfigure}[b]{0.25\linewidth}
			\centering
			\begin{tikzpicture}[scale=0.6,
				vertex/.style={circle,scale=0.6, draw=blue!60, fill=blue!10, thick, minimum size=7mm},
				edge/.style={draw=black, thick},
				labelstyle/.style={fill=white, inner sep=1pt}
				]
				
				\node[vertex] (v1) at (0,0) {$1$};
				\node[vertex] (v2) at (2,1.5) {$2$};
				\node[vertex] (v3) at (4,0) {$3$};
				\node[vertex] (v4) at (2,-2) {$4$};
				\node[vertex] (v5) at (4,2.5) {$5$};
				
				\draw[edge] (v1) to[bend left=20] node[midway,above left, labelstyle] {$w_{12}$} (v2);
				\draw[edge] (v2) to[bend left=20] node[midway, below left, labelstyle] {$w_{23}$}(v3);
				\draw[edge] (v3) -- (v4) node[midway, right, labelstyle] {$w_{34}$};
				\draw[edge] (v4) -- (v1) node[midway, below left, labelstyle] {$w_{41}$};
				\draw[edge] (v1) -- (v3) node[midway, below, labelstyle] {$w_{13}$};
				\draw[edge] (v2)  to[bend left=20] node[midway, above left, labelstyle] {$w_{25}$} (v5);
				\draw[edge] (v3) -- (v5) node[midway, below right, labelstyle] {$w_{35}$};
				
			\end{tikzpicture}
			\caption{Weighted graph}
			\label{fig:weighted}
		\end{subfigure}
		\hfill
		\begin{subfigure}[b]{0.25\linewidth}
			\centering
			\begin{tikzpicture}[scale=0.5,
				vertex/.style={circle,scale=0.6, draw=blue!60, fill=blue!10, thick, minimum size=7mm},
				edge/.style={draw=black, thick},
				labelstyle/.style={fill=white, inner sep=1pt}
				]
				
				\node[vertex] (v1) at (1,0) {$1$};
				\node[vertex] (v2) at (3,3) {$2$};
				\node[vertex] (v3) at (4,0) {$3$};
				\node[vertex] (v4) at (3,-3) {$4$};
				
				\draw[edge] (v1) to[bend left=20] node[midway, above left, labelstyle] {$w_{12} = 3$} (v2);
				\draw[edge] (v2) -- node[midway, above right, labelstyle] {$w_{23} = 1$} (v3);
				\draw[edge] (v3) to[bend left=20] node[midway, right, labelstyle] {$w_{34} = 2$} (v4);
				\draw[edge] (v4) to[bend left=15] node[midway, below left, labelstyle] {$w_{41} = 3$} (v1);
				
				
			\end{tikzpicture}
			\caption{Multigraph}
			\label{fig:multi}
		\end{subfigure}
		\hfill
		\begin{subfigure}[b]{0.48\linewidth}
			\centering
			\begin{tikzpicture}[scale=0.6, every node/.style={ draw=blue!60, fill=blue!10, scale=0.9,circle, inner sep=2pt},
				hyperedge/.style={draw=black,thick, fill=gray!30, opacity=0.3, rounded corners, inner sep=2pt}]
				
				\node (1) at (4,6) {1};  
				\node (2) at (2.5,4.5) {2};
				\node (3) at (5.5,4.5) {3};
				\node (4) at (2,2.5) {4};
				\node (5) at (4,2.5) {5};
				\node (6) at (6,2.5) {6};
				\node (7) at (8,2.5) {7};
				\node (8) at (5,1) {8};  
				\node (9) at (4,4.5) {10};  
				\node (10) at (6.5,1.75) {9};  
				
				\begin{scope}[on background layer]
					\node[hyperedge, fit=(1) (2) (3)] {};  
					\node[hyperedge, fit=(2) (4) (5)] {};  
					\node[hyperedge, fit=(3) (6) (7)] {};  
					\node[hyperedge, fit=(5) (6) (8)] {};  
					\node[hyperedge, fit=(5) (6) (3)]{};  
				\end{scope}
			\end{tikzpicture}
			\caption{$4$-uniform hypergraph on $10$ vertices}
			\label{fig:hypergraph-ex}
		\end{subfigure}
		\caption{Weighted graph, multigraph, and uniform hypergraph}
		\label{fig:three graphs}
	\end{figure}
	
	In random graphs and hypergraphs, the edges are random objects. For example, for each $2$-element subset of the vertex set, assign a Bernoulli random variable that equals $1$ if the subset forms an edge and $0$ otherwise. Similarly, in a $k$-uniform random hypergraph, each $k$-element subset of the vertex set is assigned a Bernoulli random variable indicating whether it forms a hyperedge.
	
	\subsubsection{{Erd\H os--R\'enyi Graph}}  
	Let $n$ be a natural number and let $p \in [0,1]$. The Erd\H os--R\'enyi random graph, 
	denoted $\mathcal{G}(n, p)$, is a random graph on $n$ vertices where each edge is included independently with probability $p$ \cite{erdos-reyni-1959}, \cite[Chapter-10]{alon-prob-method}.
	In other words, the  Erd\H os--R\'enyi graph $\mathcal G(n,p)$  is a 
	probability space $(\mathbb{G}(n), 2^{\mathbb{G}(n)}, \P)$, where $\mathbb{G}(n)$ denotes
	the collection of all graphs on $n$ vertices.
	The probability of a particular graph $G \in \mathbb{G}(n)$ is given by:
	$$
	\P(\{G\}) = \prod_{j \in E(G)} p \prod_{j \notin E(G)} (1-p) = p^{|E(G)|} (1-p)^{\binom{n}{2} - |E(G)|},
	$$  
	where $|E(G)|$ is the number of edges in $G$.
	
	\subsubsection{{  $k$-uniform random hypergraphs}}\label{rand-k}
	The $k$-uniform random hypergraph is a natural generalization of the Erd\H{o}s–R\'enyi graphs.
	For any integer $k \geq 2$, the \emph{$k$-uniform random hypergraph}, denoted by $H(n, k, p)$, is 
	a random hypergraph on $n$ vertices, where each $k$-set of $[n]=\{1,\ldots,n\}$ is included as a
	hyperedge with probability $p$ and independently.
	More precisely,	given integers $n \ge 2$ and $1 \le k \le n$, the  
	{\it $k$-uniform random hypergraph} $ H(n,k,p)$ is  the probability space 
	$(\mathbb{H}^k_n, 2^{\mathbb{H}^k_n}, \P)$, where $\mathbb{H}^k_n$ is the collection of all 
	$k$-uniform hypergraphs with the vertex set $[n]$, and $2^{\mathbb{H}^k_n}$ is its power set. 
	That is,  for  $H \in \mathbb{H}^k_n$ with edge set $E(H)$, the probability measure is given by
	$$
	\P(\{H\}) = \prod_{e \in E(H)} p \prod_{e \in \left[\binom{n}{k}\right] \setminus E(H)} (1 - p).
	$$
	Observe that, for $k=2$, the $k$-uniform random hypergraph $ H(n,k,p)$ coincides 
	with the Erd\H{o}s–R\'enyi graph $\mathcal G(n,p)$.

	\subsubsection{{ Linial-Meshulam complex}} Another well-studied generalization of the Erd\H{o}s–R'enyi graph to higher dimensions is the Linial–Meshulam complex, which will appear later in our discussion. 
	This model was introduced by Linial and Meshulam \cite{Linial} for dimension
	$2$ and by Wallach \cite{wallach} for higher dimensions.
	
	Let $V$ be a finite set and $\mathcal P(V)$ denote the set of all subsets of $V$. A collection $X\subset \mathcal P(V)$ is said to be a {\emph{ simplicial complex}} with vertex set $V$  if $\tau\in X$ and $\sigma\subset \tau$, then $\sigma\in X$. An element of a simplicial complex is
	called a simplex. The dimension of a simplex $\sigma$ is defined as $|\sigma|-1$, where $|\sigma|$ denotes the number of elements in $\sigma$.  An element of dimension $j$
	is called a $j$-simplex. 
	The \emph{dimension} of a non-empty simplicial complex is the highest dimension among its simplices.
	
	For a given positive integer $k\ge 2$, the Linial-Meshulam complex of 
	dimension $k$, denoted by $Y_k(n,p)$, is a random simplicial complex 
	on $n$ vertices with a complete $(k-1)$-dimensional skeleton and each
	$k$-simplices included independently with probability $p\in [0,1]$. In particular, $k=2$, the Linial Meshulam complex of dimension $1$ is the  Erd\H os-R\' enyi graph.

	Observe that the Linial-Meshulam complex of dimension $k$ is a random hypergraph where all simplices of dimension less than $(k-1)$ are included
	as hyperedges and each $ k$-dimensional simplex is included as a hyperedge
	with probability $p$ independently. However, it is not a $k$-uniform 
	random hypergraph. 
	
	\subsubsection{{\it $r$-set weighted line graph}}\label{r-set-line} Let $H=(V(H),E(H))$ be a  $k$-uniform hypergraphs on $n$ vertices. 
	For $1\le r\le k-1$, the set of all $r$-cardinality subsets of $\{1,\ldots, n\}$ is denoted by $\mathcal{V}_r=[\binom{n}{r}]$. 
	For distinct $S_1,S_2\in \mathcal V_r$, we connect $S_1$ and $S_2$ by an edge, denoted by $S_1\sim S_2$, if there exists a hyperedge $e\in E(H)$ such that $S_1,S_2\subset e$. The set of all $r$-set edges of $H$ is, denoted by $\mathcal E_r$, defined as 
	$$
	\mathcal E_r=\{(S_1,S_2)\suchthat S_1,S_2\in \mathcal V_r, S_1\sim S_2\}.
	$$
	The graph $G_r(H)=(\mathcal V_r, \mathcal E_r, w_{r,H})$ is called the {\it $r$-set weighted-line graph} of $H$, where 
	\begin{align}\label{hypw-grapphw}
		w_{r,H}(S_1,S_2)=|\{e\in E(H)\suchthat S_1,S_2\in e\}|,
	\end{align}
	where $|\cdot|$ denotes the number of elements in the set. Note that $G_r(H)$ is  a  weighted-graph  on $\binom{n}{r}$ vertices.
	For an illustration of $r$-set line graph see Figure \ref{fig:Hyp and line}.
	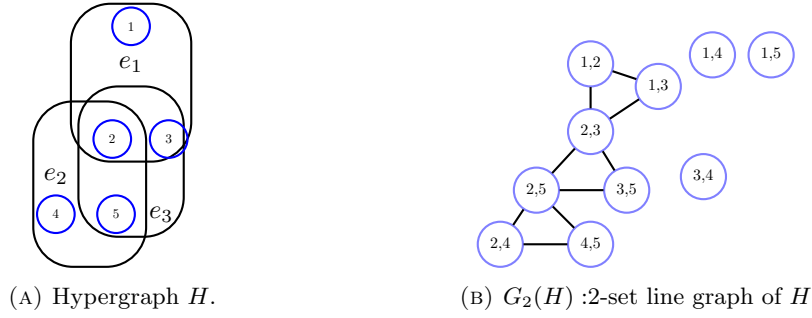
\begin{figure}[ht]
		\centering
		\begin{subfigure}[b]{0.45\linewidth}
			\centering
			\begin{tikzpicture}[
				hyperedge/.style={rounded corners=15pt, thick}]
				
				\node[scale=0.5,circle, draw=blue, thick, inner sep=0.5pt, minimum size=10mm] (v1) at (2,5) {1};
				\node[scale=0.5,circle, draw=blue, thick, inner sep=0.5pt, minimum size=10mm] (v2) at (1.75,3.5) {2};
				\node[scale=0.5,circle, draw=blue, thick, inner sep=0.5pt, minimum size=10mm] (v3) at (2.5,3.5) {3};
				\node[scale=0.5,circle, draw=blue, thick, inner sep=0.5pt, minimum size=10mm] (v4) at (1,2.5) {4};
				\node[scale=0.5,circle, draw=blue, thick, inner sep=0.5pt, minimum size=10mm] (v5) at (1.8,2.5) {5};
				\node (e1) at (2,4.5){$e_1$};
				\node (e2) at (1,3){$e_2$};
				\node (e3) at (2.4,2.5){$e_3$};
				\begin{pgfonlayer}{background}
					\draw[hyperedge] (1.2,3.2) rectangle (2.8,5.3);
					
					\draw[hyperedge] (1.3,2.2) rectangle (2.7,4.2);
					
					\draw[hyperedge] (0.7,1.8) rectangle (2.2,4);
				\end{pgfonlayer}
				
			\end{tikzpicture}
			\caption{Hypergraph $H$.}
			\label{fig:hyp-k-r}
		\end{subfigure}
		\hfill
		\begin{subfigure}[b]{0.45\linewidth}
			\centering
			\begin{tikzpicture}[scale=0.6,
				node/.style={scale=0.6,circle, draw=blue!50, thick, minimum size=10mm, inner sep=0pt},
				every path/.style={thick}
				]
				
				\node[node] (12) at (0,4) {1,2};
				\node[node] (23) at (0,2.5) {2,3};
				\node[node] (13) at (1.5,3.5) {1,3};
				\node[node] (25) at (-1.2,1.2) {2,5};
				\node[node] (35) at (0.8,1.2) {3,5};
				\node[node] (24) at (-2,0) {2,4};
				\node[node] (45) at (0,0) {4,5};
				
				\draw (12) -- (23);
				\draw (23) -- (13);
				\draw (12) -- (13);
				\draw (23) -- (25);
				\draw (23) -- (35);
				\draw (25) -- (35);
				\draw (25) -- (24);
				\draw (25) -- (45);
				\draw (24) -- (45);
				
				\node[node] (14) at (2.7,4.2) {1,4};
				\node[node] (15) at (4,4.2) {1,5};
				\node[node] (34) at (2.5,1.5) {3,4};
				
			\end{tikzpicture}
			\caption{$G_2(H):$$2$-set line graph of $H$}
			\label{fig:r-line-hyp}
		\end{subfigure}
		\caption{$H$ is a $3$-uniform hypergraph and $G_2(H)$ is the $2$-set line graph of $H$. Three hyperedges of $H$ correspond to three triangles, and the rest of the $2$-sets are isolated vertices in this $2$-set line graph.}
		\label{fig:Hyp and line}
	\end{figure}
	It is worth noting that the line graph of $Y_k(n,p)$ is defined in \cite{adhikari2023spectrumrandomsimplicialcomplexes} as the graph $(\mathcal {V}_k,E_n)$, where $\mathcal{ V}_k=[\binom{n}{k}]$ is the set of all $(k-1)$-dimensional simplices and $(\sigma, \sigma')\in E_n$ if $\sigma\cup \sigma'\in Y_k(n,p)$ for $\sigma, \sigma'\in \mathcal {V}_k$.
	Observe that, the line graph of  $Y_{k-1}(n,p)$ is the same as  the $(k-1)$-set weighted line graph of  $ H(n,k,p)$.
	\subsubsection{{ Degree of hypergraphs}}\label{exp-degree}
	Suppose  $H$ is a $k$-uniform hypergraph on $n$ vertices. For any vertex $i \in V(H)$, the \emph{star} of $i$ is the set
	$
	E_i(H) = \{e \in E(H) : i \in e\},
	$
	and the \emph{degree} of $i$, denoted by $D_i$, is defined as the cardinality of its star:
	$$
	D_i = |E_i(H)|.
	$$
	The notions of star and degree are extended for any $r$-cardinality subset of the vertex set $V(H)$. For $1\le r\le k-1$ and $S\in \mathcal V_r$,  define
	\[
	E_S(H)=\{e\in E(H):S\subseteq e\}  \mbox{ and } D_S=|E_S(H)|.
	\]
	A graph (hypergraph) is said to be a {\it locally finite graph (hypergraph respectively)} if the 
	degree of each vertex is finite. Clearly, every finite graphs and hypergraphs are locally finite. Throughout, we assume that the graphs and hypergraphs are locally finite.
	
	Observe that, in the case of $ H(n,k,p)$, the degree  $D_S$ is a Binomial random variable 
	with parameters $\binom{n-r}{k-r}$ and $p$. 
	In particular, for $r=1, k=2$ and $i\in [n]$, the distribution $D_i$ of 
	$i$-th vertex of  $\mathcal G(n,p)$ follows the Binomial random variable with 
	parameters $n-1$ and $p$. Similarly, 
	for $r=k-1$ and $\sigma\in [\binom{n}{k-1}]$, 
	the degree $D_\sigma$ of $\sigma$ in $Y_{k-1}(n,p)$ is  a Binomial random 
	variable with parameters $n-k+1$  and $p$.
	\subsection{The Local topology}\label{local-wk}
	Suppose $G$ is a finite, connected, unlabeled graph  with a distinguished vertex $o$.
	The pair $(G,o)$ is called a {\it rooted graph} with root $o$.
	We say that two rooted graphs $(G_1,o_1)$ and $(G_2,o_2)$ with $G_1=(V_1,E_1)$ 
	and $G_2=(V_2,E_2)$ are {\it isomorphic}, denoted by $G_1\cong G_2$, if there exists a bijection $f:V_1\to V_2$ 
	such that $f(o_1)=o_2$ and $\{u,v\}\in E_1$ if and only if $\{f(u), f(v)\}\in E_2$. This isomorphism leads to an equivalence relation on the set of all rooted graphs. Each equivalence class  is called an unlabeled rooted graph. We denote the collection of all unlabeled   rooted graphs as $\mathbb{G}^*$.
	
	We say that two rooted weighted graphs $(V_1,E_1,w_1,o_1)$ and $(V_2,E_2,w_2,o_2)$ are 
	isomorphic if there is an isomorphism $f:(V_1,E_1,o_1)\to (V_2,E_2,o_2)$ such that 
	$w_2(f(u),f(v)))=w_1(u,v)$ for all $u,v\in V_1$. 
	Isomorphism form an equivalence relation on the space of all weighted rooted graphs, and 
	an equivalence class under this relation is called an unlabeled weighted rooted graph.
	The collection of all unlabeled weighted  rooted graphs is denoted by $\mathbb{G}_w^*$.

	Let $H$ be a $k$-uniform hypergraph on $n$ vertices. For $1\le r\le k-1$ and  $S \subseteq V(H)$ with $|S|=r$, the pair $(H, S)$ is called an \emph{$r$-set rooted hypergraph}  with root $S$.
	For $r = 1$, a root of hypergraph is simply a vertex. 
	Two $r$-set rooted hypergraphs $(H_1,S_1)$ and $(H_2,S_2)$ are called isomorphic if there is a bijective  map $f:V(H_1)\to V(H_2)$ such that $f(e)\in E(H_2)$ if and only if $e\in E(H_1)$ and $f(S_1)=S_2$, where $f(S):=\{f(s)\suchthat s\in S\}$. The $r$-set isomorphism form an equivalence relation on the space of all 
	$r$-set rooted hypergraphs. An equivalence class is called an 
	unlabeled $r$-set rooted hypergraph. 
	
	Suppose that $(H,S)$ is an $r$-set rooted $k$-uniform hypergraph. 
	The   order pair
	$(G_r(H),S)$ is called the {\it $r$-set rooted weighted-line graph} of $(H,S)$. Observe that if
	$(H_1,S_1)$ and $(H_2,S_2)$ are 
	isomorphic then their $r$-set weighted line graphs  $(G_r(H_1),S_1)$ and $(G_r(H_2),S_2)$ are isomorphic.  
	

	\subsubsection{{Topology on $\mathbb G_w^*$}}	 Suppose $G=(V,E,w)$ is a connected  weighted 
	graph. For $u, v \in V$, the graph distance $d_G(u, v)$ 
	is defined as the length of the shortest path joining $u$ and $v$. 
	Given a vertex $v \in V(G)$ and a non-negative integer $t$, the ball of radius 
	$t$ centered at $v$ is the subgraph induced by the set
	$
	B_t(v) = \{u \in V(G) : d_G(u, v) \le t\}.
	$
	For any subset $U \subseteq V(G)$, the \emph{induced subgraph} $G_U$ is the subgraph $(U,E(G_U), w_U)$ with vertex set $U$,  edge set
	$
	E(G_U) = \{ \{u, v\} \in E(G) : u, v \in U \},
	$
	and weight $w_U=w\big |_{E(G_U)}$.
	
	For a weighted rooted graph $(G, o)$ and  $t \in \mathbb{N}$, we define  the {\it rooted $t$-neighborhood} as
	$
	(G, o)_t = (B_t(o), o).
	$
	To compare two rooted graphs $(G, u), (G', v) \in \mathbb{G}_w^*$, define the distance function $T: \mathbb{G}_w^*\to \N\cup \{0\}$ as
	$$
	T((G, u), (G', v)) = \sup \left\{ t \in \mathbb{N} \cup \{0\} : (G, u)_t \cong(G', v)_t \right\}.
	$$
	The function $T$ measures the depth up to which the two weighted rooted graphs are isomorphic.
	Using this, we define a metric $m : \mathbb{G}_w^* \times \mathbb{G}_w^* \to [0, \infty)$ defined by
	$$
	m((G, u), (G', v)) = \frac{1}{1 + T((G, u), (G', v))}.
	$$
	For $\varepsilon > 0$ and $(G, v) \in \mathbb{G}_w^*$, the open ball $B_\varepsilon((G, v))$ in this metric is given by
	$$
	B_\varepsilon((G, v)) =
	\begin{cases}
		\mathbb{G}_w^* & \text{if } \varepsilon \ge 1, \\
		\left\{ h \in \mathbb{G}_w^* : h_{\lceil 1/\varepsilon \rceil} \cong (G, v)_{\lceil 1/\varepsilon \rceil} \right\} & \text{if } \varepsilon < 1.
	\end{cases}
	$$
	The topology induced by this metric is called the \emph{local topology} on $\mathbb{G}_w^*$. It is the coarsest topology under which the map
	$$
	f((G,w, o)) = \mathbf{1}_{(G,w, o)_t \cong (G',w' v)_t}
	$$
	is continuous for every $(G', v) \in \mathbb{G}_w^*$ and each $t \ge 1$.
	Under the local topology, the space $\mathbb{G}_w^*$ is both complete and separable. This topology leads to a Borel-sigma-algebra 
	$\mathcal{B}_{m}$ on the set $\mathbb{G}_w^*$ of all wighted rooted graphs with respect to the metric $m$. We denote the collection of all the probability measure defined on $\mathcal{B}_{m}$ as $\mathcal{P}(\mathbb{G}_w^*)$. For a detailed exposition of this metric space and its topology, we refer the reader to \cite{adhikari2023spectrumrandomsimplicialcomplexes, bordenave2016spectrum} and the references therein.
	
	
	Let  $(G,w)$ be a weighted graph on $n$ vertices.
	We define a probability measure $U(G,w)$ by choosing a root 
	uniformly from the vertex set. That is
	\begin{equation}\label{UG-graph}
		U(G,w)=\frac{1}{n}\sum_{o\in V}\delta_{[(G(o),w,o)]}, 
	\end{equation}
	where $G(o)$ is the connected component of $G$ contains $o$ and $\delta_{{[(G,w,o)]}}$ is the Dirac-delta function of the rooted weighted graph $[(G,w,o)]$.

	Let $\{\nu_n\}_n$ be a sequence of measures in 
	$\mathcal{P}(\mathbb{G}_w^*)$ and $\nu$ be a measure in $\mathcal{P}(\mathbb{G}_w^*)$.
	We say that  $\{\nu_n\}_n$ {\it convergences weakly } to $\nu$, 
	denoted by $\nu_n\rightsquigarrow\nu$, if
	\begin{equation}\label{weak}
		\lim\limits_{n\to\infty}\int\limits_{\mathbb{G}_w^*}fd{\nu}_n=\int\limits_{\mathbb{G}_w^*}fd\nu
	\end{equation}
	for all function $f:\mathbb{G}_w^*\to\mathbb{R}$ which are bounded and  continuous on the metric space $(\mathbb{G}_w^*,m)$.
	A probability measure $\nu$ is called the {\it weighted local weak limit} of a sequence of graphs $\{(G_n,w_n)\}_{n\ge 1}$ if $\{U(G_n,w_n)\}$ converges weakly to $\nu$, that is,
	$$
	U(G_n,w_n)\rightsquigarrow \nu, \;\;\;\;\mbox{ as $n\to \infty$}.
	$$
	\subsection{{\it Local weak convergence of hypergraphs}}	
	For a $k$-uniform hypergraph $H$ and any integer $r$ with $1 \le r \le k-1$, 
	we consider  the $r$-set weighted line graph $G_r(H)$ of $H$. 
	We then define  
	\begin{equation}\label{u_r(H)}
		U_r(H) = U(G_r(H)).
	\end{equation} 
	In that case, $U_r(H)$ is a probability measure on $(\mathbb{G}_w^*,\mathcal{B}_m)$, i.e., an element of $\mathcal{P}(\mathbb{G}_w^*)$, arising from a uniformly chosen 
	root of $H$.
	In that case $U_r(H)$ is a probability measure on $\mathcal{P}(\mathbb{G}_w^*)$. Given a sequence of hypergraphs $\{H_n\}_{n\ge 1}$ if $\{U_r(H_n)\}$ converges weakly to $\nu$, that is,
	$$
	U_r(H_n)\rightsquigarrow \nu, \;\;\;\;\mbox{ as $n\to \infty$},
	$$
	then $\nu$ is called the {\it $r$-set local weak limit} of the sequence of hypergraphs $\{H_n\}_{n\ge 1}$. 
	
	We show that, for every $1 \le r \le k-1$, the $r$-set local weak limit of the sequence $\{H_n\}$ with
	$H_n\in H(n,k,p)$ is the $d$-block Galton-Watson measure  with $d=\binom{k}{r}-1$
	and Poisson offspring distribution almost surely. The precise definition of the $d$-block Galton–Watson measure is given in the next subsection.
	\subsection{\texorpdfstring{$d$}{d}-block Galton-Watson measure}\label{r-glt}

	\subsubsection{Galton-Watson measure} The \emph{Galton–Watson branching process} is a classical probabilistic model describing the evolution of a population in which each individual independently produces offspring according to a fixed probability distribution. The process can be represented as a random rooted tree, defined recursively by the offspring distribution.
	
	\begin{df}[Galton–Watson measure]
		Let $\mathbb{N}_0 = \mathbb{N} \cup \{0\}$, and let $\P_X$ be the law of an $\mathbb{N}_0$-valued random variable $X$ (the \emph{offspring distribution}). The \emph{Galton–Watson tree with degree distribution} $\P_X$ is the random rooted tree constructed as follows:
		\begin{enumerate}[leftmargin=*]
			\item The number of children of the root is distributed as $\P_X$.
			\item Each non-root vertex is reached via an edge from its parent. If such a vertex has total degree $k$ (including the parent edge), then it has $k-1$ children, distributed according to the \emph{shifted size-biased distribution} $\hat{\P}_X$ given by
			$$	\hat{\P}_X(k) \;=\; \frac{(k+1)\,\P_X(k+1)}{\E_{\P_X}[\mathrm{degree}]},
			\quad \text{where} \quad
			\E_{\P_X}[\mathrm{degree}] = \sum_{k \ge 0} k\,\P_X(k).$$
			\item Conditioned on their number of children, the subtrees rooted at the children of any vertex are independent and identically distributed copies of the Galton–Watson tree.
		\end{enumerate}
		The distribution of the resulting random rooted tree defines a probability measure on the space of rooted graphs supported on trees, called the \emph{Galton–Watson measure with degree distribution} $\P_X$.
	\end{df}
	
	In particular, if $X \sim \mathrm{Poisson}(\lambda)$, then $\P_X = \hat{\P}_X$, and we refer to this as the \emph{Galton–Watson measure with Poisson$(\lambda)$ offspring} (see \cite[Section~1.6]{bordenave2017mean} for details).
	Several generalizations of the Galton–Watson measure have been studied (see \cite{adhikari2023spectrumrandomsimplicialcomplexes,ZhuYizhe_Community_detection,Zhu-com}). In what follows, we introduce a particular variant, the $d$-block Galton–Watson measure, which plays a central role in our work.


	\subsubsection{ d-block Galton-Watson measure}
	For any natural number $d$, a $d$-block Galton-Watson tree is a random graph generated by the Galton-Watson branching process in which, instead of generating a single vertex, each vertex, as offspring, produces a set of $d$ vertices called a $d$-block. Thus, if an individual in the population has a total of $n$ offspring $d$-blocks, then it has a total of $nd$ offspring vertices.
	
	\begin{df}[$d$-block Galton–Watson measure with $d$-Poisson$(\lambda)$ offspring]
		The \emph{$d$-block Galton–Watson graph} is the random rooted graph generated by the following procedure:
		\begin{enumerate}[leftmargin=*]
			\item The number of \emph{$d$-block children} of each vertex follows the distribution $\P_X$ with $X \sim \mathrm{Poisson}(\lambda)$.
			\item A \emph{$d$-block child} consists of a set of $d$ new vertices. Together with their parent, these $d$ vertices form a $(d+1)$-clique (that is, a complete subgraph on $d+1$ vertices).
			\item Conditioned on the number of children, the subgraphs rooted at each child (that is, each member of the child $d$-block) are independent and identically distributed copies of the $d$-block Galton–Watson graph.
		\end{enumerate}
		The law of the random rooted graph obtained by this construction defines a probability measure on the space of rooted graphs, called the \emph{$d$-block Galton–Watson measure with $d$-Poisson$(\lambda)$ offspring}, and denoted by $\nu_{d,\lambda}$.
	\end{df}
	We now describe the rooted graphs that belong to the support of $\nu_{d,\lambda}$.
	The vertex set of such a rooted graph can be embedded into the set 
	$$\mathbb{N}^\infty_d=\left(\bigcup_{i\ge 1}(\mathbb{N}\times [d])^i\right)\cup\{o\},$$ 
	with the following ordering: $o$ is the root of the obtained graph and $o<\ii$ for all $\ii\ne o\in $. 
	For $(\ii,\ii'),(\jj,\jj')\in\mathbb{N}^\infty_d$, suppose that 
	$(\ii,\ii')\in (\mathbb{N}\times [d])^p,(\jj,\jj')\in (\mathbb{N}\times [d])^q$. 
	We set $(\ii,\ii')<(\jj,\jj')$ if
	\begin{enumerate}[leftmargin=*]
		\item $p<q$,
		\item $p=q$, and $(\ii,\ii')=((i_1,i_1'),\ldots,(i_p,i_p'))$, $(\jj,\jj')=((j_1,j_1'),\ldots,(j_p,j_p'))$. Furthermore
		$r=\inf\{s:(i_s,i'_s)\ne(j_s,j'_s)\}$ with one of the followings holds:
		\begin{enumerate}
			\item $i_s<j_s$ 
			\item $i_s=j_s$, and $i'_s<j'_s$ with respect to the ordering in $[d]$.
		\end{enumerate}
	\end{enumerate}
	The only edges are obtained by forming $(d+1)$-clique containing a $d$-block offspring along with its parent.

	Let $\{X_{i,n}\}$ be a collection of i.i.d.\ random variables with common distribution $\P_X$ such that $X \sim \mathrm{Poisson}(\lambda)$. Define the sequence $\{Z_n\}_{n \geq 0}$ by setting
	$$
	Z_0 = 1, \quad \text{and} \quad Z_{n+1} = \sum_{i=1}^{Z_n}d X_{i,n}, \quad \text{for all } n \geq 0.
	$$
	Here $Z_n$ represents the number of individuals in the $n$-th generation, and $X_{i,n}$ denotes the number of $d$-block offspring of the $i$-th individual in that generation. The offspring of an individual $i$ in generation $n$ are labeled by $$(i,1,1),\ldots,(i,1,d), (i,2,1),\ldots, (i,2,d), \ldots, (i,X_{i,n},1),\ldots,(i,X_{i,n},d).$$ 
	We have illustrated a $2$-block Galton-Watson tree in Figure \ref{fig:3-glt}.
	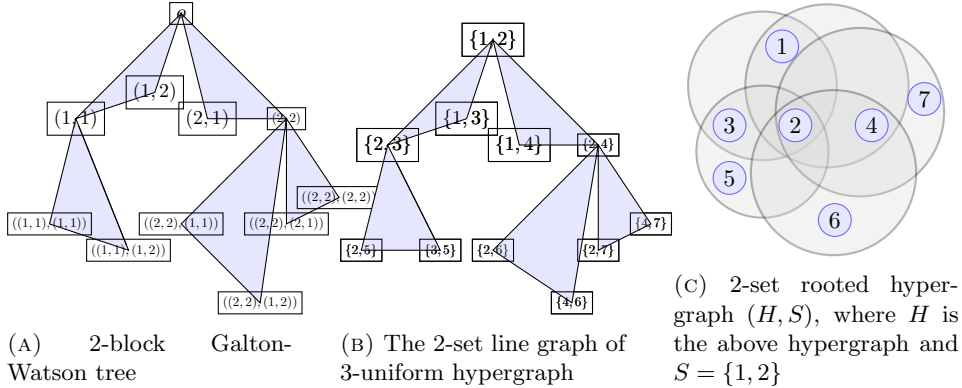
\begin{figure}[ht]
		\centering
		\begin{subfigure}[b]{0.3\linewidth}
			\centering
			
			\begin{tikzpicture}[scale=0.7]
				
				\node[scale=.75] (R) at (0,0) [] {};
				
				\node[scale=0.75] (A) at (-2,-2) [] {};
				\node[scale=0.75] (B) at (-0.5,-1.5) [] {};
				\node[scale=0.75] (C) at (.5,-2) [] {};
				\node[scale=0.5] (D) at (2,-2) [] {};
				
				\node[scale=0.5] (A1) at (-2.5,-4) [] {};
				\node[scale=0.5] (A2) at (-1,-4.5) [] {};
				
				\node[scale=0.5] (D1) at (0,-4) [] {};
				\node[scale=0.5] (D2) at (1.5,-5.5) [] {};
				\node[scale=0.5] (D3) at (2,-4) [] {};
				\node[scale=0.5] (D4) at (3,-3.5) [] {};
				\draw[fill=blue!10] (R.center)
				to (A.center)
				to (B.center)
				to cycle;
				
				\draw[fill=blue!10] (R.center)
				to (C.center)
				to (D.center)
				to cycle;

				\draw[fill=blue!10] (A.center)
				to (A1.center)
				to (A2.center)
				to cycle;
				\draw (A) -- (A2);
				
				\draw[fill=blue!10] (D.center)
				to (D1.center)
				to (D2.center)
				to cycle;
				\draw[fill=blue!10] (D.center)
				to (D3.center)
				to (D4.center)
				to cycle;
				\node[scale=.75] (R) at (0,0) [rectangle, draw] {$o$};
				
				\node[scale=0.75] (A) at (-2,-2) [rectangle, draw] {$(1,1)$};
				\node[scale=0.75] (B) at (-0.5,-1.5) [rectangle, draw] {$(1,2)$};
				\node[scale=0.75] (C) at (.5,-2) [rectangle, draw] {$(2,1)$};
				\node[scale=0.5] (D) at (2,-2) [rectangle, draw] {$(2,2)$};
				
				\node[scale=0.5] (A1) at (-2.5,-4) [rectangle, draw] {$((1,1),(1,1))$};
				\node[scale=0.5] (A2) at (-1,-4.5) [rectangle, draw] {$((1,1),(1,2))$};
				
				\node[scale=0.5] (D1) at (0,-4) [rectangle, draw] {$((2,2),(1,1))$};
				\node[scale=0.5] (D2) at (1.5,-5.5) [rectangle, draw] {$((2,2),(1,2))$};
				\node[scale=0.5] (D3) at (2,-4) [rectangle, draw] {$((2,2),(2,1))$};
				\node[scale=0.5] (D4) at (3,-3.5) [rectangle, draw] {$((2,2),(2,2))$};
			\end{tikzpicture}
			\caption{$2$-block Galton-Watson tree}
			\label{fig:3-glt}
		\end{subfigure}
		\hfill
		\begin{subfigure}[b]{0.3\linewidth}
			\centering
			\begin{tikzpicture}[scale=0.7]
				
				\node[scale=.75] (R) at (0,0) [rectangle, draw] {$\{1,2\}$};
				
				\node[scale=0.75] (A) at (-2,-2) [rectangle, draw] {$\{2,3\}$};
				\node[scale=0.75] (B) at (-0.5,-1.5) [rectangle, draw] {$\{1,3\}$};
				\node[scale=0.75] (C) at (.5,-2) [rectangle, draw] {$\{1,4\}$};
				\node[scale=0.5] (D) at (2,-2) [rectangle, draw] {$\{2,4\}$};
				
				\node[scale=0.5] (A1) at (-2.5,-4) [rectangle, draw] {$\{2,5\}$};
				\node[scale=0.5] (A2) at (-1,-4) [rectangle, draw] {$\{3,5\}$};
				
				\node[scale=0.5] (D1) at (0,-4) [rectangle, draw] {$\{2,6\}$};
				\node[scale=0.5] (D2) at (1.5,-5) [rectangle, draw] {$\{4,6\}$};
				\node[scale=0.5] (D3) at (2,-4) [rectangle, draw] {$\{2,7\}$};
				\node[scale=0.5] (D4) at (3,-3.5) [rectangle, draw] {$\{4,7\}$};
				\draw[fill=blue!10] (R.center)
				to (A.center)
				to (B.center)
				to cycle;
				
				\draw[fill=blue!10] (R.center)
				to (C.center)
				to (D.center)
				to cycle;

				\draw[fill=blue!10] (A.center)
				to (A1.center)
				to (A2.center)
				to cycle;
				\draw (A) -- (A2);
				
				\draw[fill=blue!10] (D.center)
				to (D1.center)
				to (D2.center)
				to cycle;
				\draw[fill=blue!10] (D.center)
				to (D3.center)
				to (D4.center)
				to cycle;
				\node[scale=.75] (R) at (0,0) [rectangle, draw] {$\{1,2\}$};
				
				\node[scale=0.75] (A) at (-2,-2) [rectangle, draw] {$\{2,3\}$};
				\node[scale=0.75] (B) at (-0.5,-1.5) [rectangle, draw] {$\{1,3\}$};
				\node[scale=0.75] (C) at (.5,-2) [rectangle, draw] {$\{1,4\}$};
				\node[scale=0.5] (D) at (2,-2) [rectangle, draw] {$\{2,4\}$};
				
				\node[scale=0.5] (A1) at (-2.5,-4) [rectangle, draw] {$\{2,5\}$};
				\node[scale=0.5] (A2) at (-1,-4) [rectangle, draw] {$\{3,5\}$};
				
				\node[scale=0.5] (D1) at (0,-4) [rectangle, draw] {$\{2,6\}$};
				\node[scale=0.5] (D2) at (1.5,-5) [rectangle, draw] {$\{4,6\}$};
				\node[scale=0.5] (D3) at (2,-4) [rectangle, draw] {$\{2,7\}$};
				\node[scale=0.5] (D4) at (3,-3.5) [rectangle, draw] {$\{4,7\}$};
			\end{tikzpicture}
			\caption{The $2$-set line graph of $3$-uniform hypergraph}
			\label{fig:glt-pic}
		\end{subfigure}
		\hfill
		\begin{subfigure}[b]{0.3\linewidth}
			\centering
			\begin{tikzpicture}[scale=0.7, every node/.style={ draw=blue!60, fill=blue!10, scale=0.9,circle, inner sep=2pt},
				hyperedge/.style={draw=black,thick, fill=gray!30, opacity=0.3, rounded corners, inner sep=2pt}]
				
				\node (1) at (4,6) {1};  
				\node (3) at (3,4.5) {3};
				\node (4) at (5.7,4.5) {4};
				\node (6) at (5,2.7) {6};
				\node (5) at (3,3.5) {5};
				\node (7) at (6.7,5) {7};
				\node (2) at (4.25,4.5) {2};  
				
				\begin{scope}[on background layer]
					\node[hyperedge, fit=(1) (2) (3)] {};  
					\node[hyperedge, fit=(1) (2) (4)] {};  
					\node[hyperedge, fit=(3) (2) (5)] {};  
					\node[hyperedge, fit=(2) (4) (7)] {};  
					\node[hyperedge, fit=(2) (6) (4)]{};  
				\end{scope}
			\end{tikzpicture}
			\caption{$2$-set rooted hypergraph $(H,S)$, where $H$ is the above hypergraph and $S=\{1,2\}$}
			\label{fig:HS)}
		\end{subfigure}
		\caption{A $2$-set rooted hypergraph and its $2$-set line graph.}
		\label{fig:d}
	\end{figure}
	We refer the reader to \cite[Section 4.2]{adhikari2023spectrumrandomsimplicialcomplexes} for a detailed description of $ d$-block Galton-Watson tree and the corresponding measure. Since $\nu_{d,\lambda}\in \mathcal{P}(\mathbb{G}^*)$, that is a probability measure on the collection of all rooted graphs, $\nu_{d,\lambda}$ can be represented as a random rooted graph, we denote this random rooted graph by $(GW_d,o)$.
	

	For where we set $d = \binom{k}{r} - 1$, the $(\binom{k}{r} - 1)$-block Galton–Watson tree $(GW_d, o)$ admits a natural representation as the $r$-set line graph of an $r$-set rooted $k$-uniform hypergraph $(T_k, S)$, where $|S|=r$. A root $o$ corresponds to an $r$-element subset $S$ of the vertex set $V(H)$. More generally, each vertex $v$ in $GW_d$ corresponds to some $r$-cardinality subset $S \subset V(H)$. 
	Every $(\binom{k}{r} - 1)$-block of offspring of $v$ arises from a hyperedge $e$ of $H$ that contains $S$. Since $e$ consists of $k$ vertices, it contains exactly $\binom{k}{r}$ distinct $r$-subsets. Among these, one is $S$ itself, while the remaining $\binom{k}{r} - 1$ subsets correspond to the vertices in the offspring $(\binom{k}{r} - 1)$-block  of $v$. Thus, $(GW_d, o)$ is the $r$-set line graph of an $r$-set rooted $k$-uniform hypergraph $(T_k, S)$. A $3$-uniform hypergraph $H$ and its $2$-set underlying graph is a $2$-block Galton-Watson tree illustrated in Figure \ref{fig:d}.

	\subsubsection{Structure of $(T_k,S_r)$}\label{GLTH}
	We begin with the rooted $d$-block Galton–Watson graph $(GW_d,o)$, where 
	$
	d = \binom{k}{r} - 1.
	$
	From this, we construct an associated rooted $k$-uniform hypergraph $(T_k,S_r)$ as follows.
	
	Let the root $o$ correspond to an $r$-element subset $S_r$. Each $d$-block offspring of $o$ is associated with a $k$-uniform hyperedge $e$ containing $S_r$. Since $e$ contains exactly $\binom{k}{r}$ distinct $r$-subsets, one of which is $S_r$, the remaining $\binom{k}{r}-1 = d$ number of $r$-subsets are put in bijection with the $d$ vertices of the offspring block of $o$. If $e_i$ and $e_j$ are two distinct hyperedges containing $S_r$ (corresponding to two distinct $d$-block offspring of $o$), then by construction $e_i \cap e_j = S_r$. 
	
	Each offspring vertex of $o$ thus corresponds to an $r$-subset. Repeating the same procedure for each such vertex—using its associated $r$-subset as the “root” $S_r$ in the next step—yields the $k$-uniform hypergraph $T_k$.
	\begin{exm}\rm
		Consider the rooted graph $(G,o)$ in Figure \ref{fig:3-glt} and its associated $3$-uniform hypergraph $H$ in Figure \ref{fig:d}c. Let the root $o$ correspond to the $2$-subset $\{1,2\}$. Its $2$-block offspring $\{(1,1),(2,2)\}$ corresponds to the hyperedge $e_1 = \{1,2,3\}$. The other two $2$-subsets of $e_1$—namely $\{2,3\}$ and $\{1,3\}$—are associated with the vertices $(1,1)$ and $(1,2)$ of $G$, respectively. 
		Similarly, the $2$-block offspring $\{(2,1),(2,2)\}$ of $o$ is associated with the hyperedge $e_2 = \{1,2,4\}$. The $2$-block offspring $\{((1,1),(1,1)), ((1,1),(1,2))\}$ of $(1,1)$ is associated with the hyperedge $e_3 = \{2,3,5\}$, which contains the $2$-set $\{2,3\}$ corresponding to $(1,1)$. Likewise, the $2$-block offspring $\{((2,2),(1,1)), ((2,2),(1,2))\}$ of $(2,2)$ corresponds to $e_4 = \{2,4,6\}$, and the $2$-block offspring $\{((2,2),(2,1)), ((2,2),(2,2))\}$ corresponds to $e_5 = \{2,4,7\}$.
	\end{exm}

	\subsection{The \texorpdfstring{$r$}{r}-set perspective and related hypergraph matrices.} 
	
	This work is motivated in part by the study of $r$-set random walks on $k$-uniform hypergraphs. For a $k$-uniform hypergraph $H$, the $r$-set random walk is the natural random walk on $G_r(H)$, where transitions occur between $r$-subsets that co-occur in a hyperedge of $H$. For any fixed $r < k$, one may consider $r$-element subsets of the vertex set as fundamental units or \emph{roots}, rather than individual vertices. The walk transitions from $S_1$ to $S_2$ with positive probability if $S_1 \cup S_2 \subseteq e$ for some $e \in E(H)$; in other words, $S_1$ and $S_2$ are adjacent in $G_r(H)$.
	This $r$-set perspective appears in several contexts, particularly in the study of random walks, the construction of adjacency-type matrices, and spectral methods for hypergraphs \cite{naik2019intersection-r-set, intersection-hypergraph-algo, r-walk}.
	Formally, the $r$-set random walk on a $k$-uniform hypergraph is equivalent to an weighted variation of random walk on its associated $r$-set-weighted line graph. Notably, in the special case when $k = 2$ and $r = 1$, this construction reduces to the classical simple random walk on Erdős–Rényi random graphs. The Linial–Meshulam complex \cite{Linial} can also be viewed as a special case of this perspective, corresponding to $r = k-1$.
	
	
	\subsubsection{Adjacency operators} 
	Let $(G,w)$ be a locally finite weighted graph with at most countably many vertices; denote the vertex set by $V(G)$. 
	Define
	$$
	\ell^2(V(G))=\{f:V(G)\to\mathbb{C}\mid \sum_{v\in V(G)}|f(v)|^2<\infty\}.
	$$
	For each $v\in V(G)$ let $\mathbf{1}_v:V(G)\to\mathbb{R}$ be the point mass at $v$, i.e.\ $\mathbf{1}_v(u)=1$ if $u=v$ and $\mathbf{1}_v(u)=0$ otherwise. 
	Let $\ell_c(V(G))$ denote the linear span of these point masses, i.e.\ the space of finitely supported functions on $V(G)$. We view $\ell^2(V(G))$ as a Hilbert space equipped with the $\ell^2$-norm
	$$
	\|f\|_2 = \left(\sum_{v\in V(G)} |f(v)|^2 \right)^{1/2}.
	$$
	Note that $\ell_c(V(G))$ is dense in $\ell^2(V(G))$ with respect to  this $\ell^2$-norm.
	Define the weighted adjacency operator $A_G^w$ on $\ell_c(V(G))$ by
	\begin{equation}\label{adj-graph-w}
		(A_G^w f)(v)=\sum_{u\in V(G)} w(v,u)\,f(u),\qquad v\in V(G).
	\end{equation}
	Here the weight is symmetric that is $w(u,v)=w(v,u)$for any pair of vertices $u$ and $v$.
	If $f\in\ell_c(V(G))$, write $S_f=\operatorname{supp}(f)$; then $S_f$ is finite and $f(u)=0$ for all $u\notin S_f$. Hence for each $v\in V(G)$,
	$(A_G^w f)(v)$ can be nonzero only if there exists $u\in S_f$ with $w(v,u)\neq 0$. Therefore,
	$$
	\operatorname{supp}(A_G^w f)\subseteq \bigcup_{u\in S}\{v\in V(G)\mid w(v,u)\neq 0\}.
	$$
	Under the local finiteness hypothesis (meaning for each vertex $x$ the set $\{y\in V(G)\mid w(x,y)\neq 0\}$ is finite), each set on the right is finite, and a finite union of finite sets is finite. Thus $A_G^w f$ is finitely supported, i.e.\ $A_G^w f\in\ell_c(V(G))$. In particular
	$$
	A_G^w:\ell_c(V(G))\to\ell_c(V(G))
	$$
	is well defined, and since $\ell_c(V(G))$ is dense in $\ell^2(V(G))$, $A_G^w$ is densely defined on $\ell^2(V(G))$ with domain $\ell_c(V(G))$.
	Since $w$ is symmetric, the operator $A_G^w$ is symmetric on $\ell_c(V(G))$, hence densely defined and formally self-adjoint. If the vertex degrees are uniformly bounded and the weights are bounded, then $A_G^w$ extends to a bounded self-adjoint operator on $\ell^2(V(G))$. 
	\subsubsection{\texorpdfstring{$r$}{r}-set adjacency operator}
	The $r$-set weighted line graph associated with a hypergraph $H$ is the weighted graph $G_r(H) = (\mathcal{V}_r, \mathcal{E}_r, w_{r,H})$, as defined in Section \ref{r-set-line}. We shall refer to the adjacency matrix of $G_r(H)$ as the \emph{$r$-set adjacency matrix} of the hypergraph $H$. As in the case of weighted graphs (see previous subsection), $A_H^{w_r}$ is initially defined on $\ell_c(\mathcal{V}_r)$, making it a densely defined operator on $\ell^2(\mathcal{V}_r)$.
	Suppose that $H$ is a locally finite $k$-uniform hypergraph. 
	The space  of finitely supported functions on $\mathcal{V}_r$ is 
	denoted by $\ell_c(\mathcal{V}_r)$, that is,
	$$
	\ell_c(\mathcal{V}_r) = \{ f : \mathcal{V}_r \to \mathbb{C} \mid f(S) \neq 0 \text{ for finitely many } S \in \mathcal{V}_r \},
	$$
	and its completion $\ell^2(\mathcal{V}_r)$ consisting of square-summable functions
	$$
	\ell^2(\mathcal{V}_r) = \left\{ f : \mathcal{V}_r \to \mathbb{C} \mid \sum_{S \in \mathcal{V}_r} |f(S)|^2 < \infty \right\}.
	$$
	This Hilbert space is equipped with the inner product
	$$
	\langle f, g \rangle_r := \sum_{S \in \mathcal{V}_r} f(S) \overline{g(S)}.
	$$
	The $r$-set adjacency operator $A_H^{w_r} : \ell^2(\mathcal{V}_r) \to \ell^2(\mathcal{V}_r)$ is defined by
	\begin{align}\label{eqn:A_H}
		(A_H^{w_r} f)(S) = \sum_{S' \in \mathcal{V}_r} w_{r,H}(S, S') f(S'),
	\end{align}
	Since the hypergraph $H$ is locally finite, each $r$-set $S$ participates in only finitely many hyperedges, so the set of $S'$ with $w_{r,H}(S,S') \neq 0$ is finite. This ensures the sums in \eqref{eqn:A_H} are finite.
	The weight $w_{r,H}(S,S')$ is defined in \eqref{hypw-grapphw}, 
	and $w_{r,H} : \mathcal{V}_r \times \mathcal{V}_r \to \mathbb{R}$ is a weight function satisfying $w_{r,H}(S, S') = w_{r,H}(S', S)$ and $w_{r,H}(S,S')=0$ if $S\nsim S'$ (that is, there does not exists $e\in E(H)$ such that $S,S'\subset e$). The weight function $w_{r,H}$ respect each $r$-set rooted hypergraph isomorphism. That is, if $\sigma:V(H)\to V(H')$ is an $r$-set rooted hypergraph isomorphism, then $w_{r,H}(S,S')=w_{H'_r}(\sigma(S),\sigma(S'))$ for all $r$-cardinality subset $S,S'$ of $V(H)$. For a finite $k$-uniform hypergraph $H$ on $n$ vertices, $A_H^{w_r}$ is represented by an $\binom{n}{r} \times \binom{n}{r}$ matrix. When $H \in H(n, k, p)$ is a random hypergraph, $A_{H}^{w_r}$ becomes a random matrix; we discuss its spectral properties further in this section.
	
	In particular, for $r=1$ the operator $A_H^{w_r}$ becomes the weighted adjacency operator $A^w_H:\ell_c(V(H))\to \ell_c(V(H))$ is such that
	\begin{equation*}
		(A^w_Hf)(v)=\sum\limits_{u\in V(G)}w_H(v,u)f(u),
	\end{equation*}
	where  weight function $w_H$ is as defined in \eqref{hypw-grapphw}
	and for any rooted hypergraph isomorphism $\sigma:V(H)\to V(H')$, we have $w_H(i,j)=w_{H'}(\sigma(i),\sigma(j))$ for any pair $i,j\in V(H)$.
	
	
	\subsection{The Spectrum of adjacency matrices} 
	
	\subsubsection{ Limiting spectral distribution} 
	For an $n \times n$ symmetric matrix $M_n$ with 
	eigenvalues $\lambda_1, \lambda_2, \ldots, \lambda_n$, 
	the \emph{empirical spectral distribution} is defined by
	$$
	F_{M_n}(x) = \frac{1}{n} \sum_{i=1}^n \mathbf{1}(\lambda_i \leq x).
	$$
	The \emph{empirical spectral measure} of $M_n$, is given by
	$$
	\mu_{M_n} = \frac{1}{n} \sum_{i=1}^n \delta_{\lambda_i},
	$$
	where $\delta_{\lambda_i}$ denotes the Dirac measure at $\lambda_i$. A probability measure $\mu$ is said to be the {\it limiting spectral measure} of $M_n$ if $\mu_{M_n}$ converges to $\mu$ weakly, that is,
	\[
	\int f d\mu_{M_n}\to \int f d\mu 
	\]
	for all bounded continuous functions $f$ on $\R$. Its distribution is known as the {\it limiting spectral distribution (LSD)} of the sequence of matrices $M_n$. Alternatively, a function $F$ is said to be the LSD of $\{M_n\}$  if $F_n(x)$ converges to $F(x)$, as $n\to \infty$, for all continuity points $x$ of $F$. 

	\subsubsection{Spectral measure of adjacency operator of weighted adjacency operator}\label{spectra-adj}
	For any graph $G$, the collection of square sumable function $\ell^2(V(G))$ is a Hilbert space with usual inner product
	\(
	\langle f, g\rangle_{V(G)} =\sum_{v\in V(G)}f(v)\overline{g(v)},
	\)
	$\mbox{ for } f,g\in \ell^2(V)$.
	That is, each $(G,o)\in\mathbb{G}^*_w$, corresponds to a Hilbert space $\ell^2(V(G))$. To incorporate the collection of Hilbert space $\{\ell^2(V(G)):(G,o)\in\mathbb{G}^*_w\}$ into a single Hilbert space we recall some notion from \cite[Section-5, P.27]{Aldous-Russell}. Using any $\rho\in \mathcal P_{uni}(\mathbb G^*_w)$ the global Hilbert space $\mathcal{H}$
	is defined as the following  a \emph{direct integral}: $$\mathcal{H}=\int^\oplus\ell^2(V(G))d\rho(G,o).$$
	Each vector $\mathbf{f}\in \mathcal{H}$ associate each rooted hypergraph $(G,o)$ to a vector $\mathbf{f}_{(G,o)}\in \ell^2(V(G))$. This fact is represented using the following notation:
	$$ \mathbf{f}=\int^\oplus \mathbf{f}_{(G,o)}d\rho(G,o).$$
	Using the inner products $\langle\cdot,\cdot \rangle_{(V(G))}$ for all Hilbert spaces $\ell^2(V(G))$, we define an inner product $\E_\rho\langle\cdot,\cdot \rangle_{\mathcal{H}}$ on $\mathcal H$ as
	$$\E_\rho\langle\mathbf{f},\mathbf{g} \rangle_{\mathcal{H}}=\int(\langle \mathbf{f}_{(G,o)},\mathbf{g}_{(G,o)}\rangle_{(V(G))})d\rho(G,o).$$
	A graph operator $A=\int^\oplus A{(G,o)}d\rho(G,o)$ acting on $\mathcal{H}$ is such that for each $(G,o)\in\mathbb{G}^*_w$, $A$ induces an operator $A(G)$ acting on $\ell^2(V(G))$. Let $\mathcal{M}$ be the \emph{von Neumann algebra} of ($\rho$-equivalence class of) such operators $A$ acting on $\mathcal{M}$. Each $A\in\mathcal{M}$ are equivariant in the sense that for all weighted graph isomorphism $\phi:V(G_1)\to V(G_2)$ , the inner product $\langle A(G_1,o)\mathbf{1}_u,\mathbf{1}_v \rangle_{(V(G_1))}=\langle A(G_2,\phi(o))\mathbf{1}_{\phi(u)},\mathbf{1}_{\phi(v)} \rangle_{(V(G_1))}$.

	For a matrix $A$ associated with graph, the above condition ensures that $A(G_1)_{uv}=A(G_2)_{\phi(u)\phi(v)}$ for any two vertices $u,v\in V(G_1)$. The von Neumann algebra $\mathcal{M}$ is endowed with a trace operator $\Tr_{\rho}:\mathcal{M}\to\mathbb{C}$ defined as
	$$\Tr_{\rho}(A)=\E_\rho\langle A\mathbf{1}_o,\mathbf{1}_{o}\rangle\mbox{~for all~}A\in\mathcal{M}.$$
	In a von Neumann Algebra, a trace operator is a linear, non-negative valued and tracial, that is for two element $A$ and $B$ in the von Neumann Algebra with trace operator $\Tr_{\rho}$ is such that $Tr(AB)=Tr(BA)$.
	It can be easily verified the operator $\Tr_{\rho}$ is a linear operator. For any $A\in\mathcal{M}$, we have
	\begin{align*}
		\Tr_{\rho}(A)=\E_\rho\langle A\mathbf{1}_o,\mathbf{1}_{o}\rangle=\int\langle A(G,o)\mathbf{1}_o,\mathbf{1}_o\rangle d\rho(G,o).
	\end{align*}
	Since $\langle A(G,o)\mathbf{1}_o,\mathbf{1}_o\rangle\ge 0$, the trace $\Tr_{\rho}(A)\ge0$ for all $A\in\mathcal{M}$.
	
	A closed densely defined operator $T$
	is said to be \emph{affiliated with} $\mathcal{M}$ if it commutes with all unitary operators that commute with all the elements of $\mathcal{M}$. If $A \in \mathcal M$ is self-adjoint, the spectral theorem gives a projection-valued measure 
	$E^{A}(\cdot)$ so that
	$$
	A = \int_{\mathbb{R}} x \, dE^{A}(x).
	$$
	Thus, for $A\in\mathcal{M}$ and a unit vector $\mathbf{f}\in \mathcal{H}$, we can define a probability measure $\mu_A^\mathbf{f}$ on $\mathbb{R}$ such that 
	\begin{equation}\label{spectral-vector-measure}
		\mu_A^\mathbf{f}(B)=\int_B  \langle dE^A(x)\,\mathbf{f}, \mathbf{f} \rangle.
	\end{equation}
	This measure $\mu_A^\mathbf{f}$ is referred to as the \emph{spectral measure of} $A$ associated with the vector $\mathbf{f}$.
	We define a probability measure 
	$\mu_\rho$ on $\mathbb{R}$ by
	\begin{equation}\label{mean-spec-m}
		\mu_\rho(B) := \operatorname{Tr}_\rho\big(E^A(B)\big) =\E_\rho[\mu_A^{\mathbf{1}_o}](B)
		= \E_\rho\left[ \int_B \langle dE^A(x)\,\mathbf{1}_o, \mathbf{1}_o \rangle \right],
	\end{equation}
	for all Borel sets $B \subset \mathbb{R}$.
	This is the spectral distribution of $A$ under the state $\rho$. Then for each integer $k \ge 0$,
	$$
	\operatorname{Tr}_\rho\!\left(A^k\right)
	= \int x^k \, d\mu_A(x).
	$$
	Let \( E^A(\{0\}) \) be the spectral projection of \( A \) at \( 0 \). Then
	$$
	\mu_\rho(\{0\}) = \operatorname{Tr}_\rho\big(E^A(\{0\})\big)
	$$
	is the \emph{nullity} of $ A $ (the mass at zero), i.e., the normalized size of $ \ker A $. The support projection of $ A $ is
	$$
	s(A) = 1 - E^A(\{0\}),
	$$
	and its trace
	$$
	\operatorname{Tr}_\rho\big(s(A)\big) = 1 - \mu_\rho(\{0\})
	$$
	measures the \emph{rank} of \( A \). Thus the definition
	$$
	\operatorname{rank}(A) := 1 - \mu_\rho(\{0\})
	$$
	is the analogue of the rank–nullity formula, where
	\begin{equation}\label{rank-null}
		\operatorname{rank}(A) + \mu_\rho(\{0\}) = 1.
	\end{equation}
	In other words “the normalized rank plus nullity equal one.”

	
	
	Recall, $\nu_{d,\lambda}$ denotes the measure induced by the $d$-block Galton-Watson process with Poisson offspring distribution. Since the measure $\nu_{d,\lambda}$ belongs to $\mathcal{P}(\mathbb{G}^*_w)$, the measure $\mu_{\nu_{d,\lambda}}$ is well defined. We show that under a suitable condition the limiting spectral measure of the $r$-adjacency matrices of $k$-uniform random hypergraphs is $\mu_{\nu_{d,\lambda}}$ with $d=\binom{k}{r}-1$.  See Theorem \ref{spe-conv}.

	
	\section{Main results}\label{se:mainresult}
	We state our main results in this section.  Recall that $H(n,k,p)$ denotes the $k$-uniform random hypergraph on $n$ vertices, where each hyperedge is included independently with probability $p$. We also recall that  $\nu_{d,\lambda}$ denotes the probability measure on $\mathcal G^*_w$ induced by the $d$-block Galton–Watson process with Poisson$(\lambda)$ offspring distribution. We have the following assumption on the parameters.
	\begin{assumption}\label{assum} Fix $\lambda>0$ and $k\in \N$. For $1\le r\le k-1$,  we have $n\in \N$ and $p>0$ (depending on $\lambda
		,k, r, n$) such that 
		$$
		\binom{n-r}{k-r}p=\lambda.
		$$
	\end{assumption}
	Our first result is a natural extension of the local convergence of Er\H os-R\' enyi graphs in the $k$-uniform hypergraphs setup. Specifically, the result provides the 1-set local weak limit of $k$-uniform random hypergraphs. In this case, a root refers to a vertex of the hypergraphs.
	
	\begin{thm}\label{deviationto0}
		Let $r=1$ and $\lambda,k,n,p$ be as in Assumption \ref{assum} and $H_n\in H(n,k,p)$. Then the $1$-set local weak limit of $H_n$ is  $\nu_{d, \lambda}$ with $d=k-1$ almost surely. That is,
		\[
		U_1(H_n) \rightsquigarrow \nu_{k-1,\lambda}, \mbox{ as $n\to \infty$,  }
		\]
		almost surely.
	\end{thm} 
	Observe that, for  $k = 2$, the $1$-set weighted line graph of $H(n, k, p)$ reduces to the classical Erdős–Rényi random graph. In this special case, our main result yields the following corollary,
	which is a well-known result and  is also known as  the Benjamini-Schramm convergence for the Erdős–Rényi graphs.
	\begin{cor}\label{cor-Benjamini and Schramm}
		Let  $\lambda>0$ and   $G_n\in \mathcal G(n,p)$,  the Erdős–Rényi graph. If  $np\to \lambda$ as $n\to \infty$ then the local weak limit of $G_n$ converges to  $\nu_{1,\lambda}$ almost surely.  That is, 
		\[
		U_1(G_n) \rightsquigarrow \nu_{1,\lambda}, \mbox{ as $n\to \infty$,  }
		\]
		almost surely.
	\end{cor}
	The next result extends  Theorem \ref{deviationto0}.
	It provides the $r$-set local weak convergence of $k$-uniform random hypergraphs for $1\le r\le k-1$. 
	
	\begin{thm}\label{main}
		Let $\lambda,k,r,n,p$ be as in Assumption \ref{assum} and $H_n\in H(n,k,p)$. Then the $r$-set local weak limit of $H_n$ is  $\nu_{d, \lambda}$ with $d=\binom{k}{r}-1$ almost surely. That is,
		$$
		U_r(H_n) \rightsquigarrow \nu_{\binom{k}{r}-1,\lambda}, \mbox{ as $n\to \infty$.  }
		$$
		almost surely.
	\end{thm}
	Though Theorem \ref{deviationto0} follows directly from Theorem \ref{main}, we provide a separate proof of this result to enhance the readability of this work. 
	For $r=k-1$, the $(k-1)$-set rooted $k$-uniform hypergraph resembles the Linial-Meshulam model \cite{Linial,Meshulam}. 	The notion of the $k$-set local limit of $H(n,k+1,p)$ matches with the notion of the local weak limit of the line graphs of the Linial-Meshulam model of dimension $k$, as defined in \cite{adhikari2023spectrumrandomsimplicialcomplexes}. Thus, the following Corollary of Theorem \ref{main} gives  \cite[Theorem 11]{adhikari2023spectrumrandomsimplicialcomplexes}.
	\begin{cor}\label{Linial-Meshulam}
		Let $\lambda>0$ and $ Y_k(n,p)$ be the Linial-Meshulam complexes of dimension $k$ on $n$ vertices. Suppose  $np\to\lambda$ as $n\to \infty$. 
		Then the local weak limit of the line graph of $Y_k(n,p)$ is  $\nu_{k,\lambda}$ almost surely. 
	\end{cor}
	The study of the $r$-set local weak limit is partly originated to study the spectral properties of the $r$-set adjacency matrices of the $k$-uniform random hypergraphs. To our knowledge, the spectral theory on hypergraphs is much less explored than on graphs. 
	Next result provides the LSD of the $r$-set adjacency matrices related to the $r$-random walk on $H(n,k,p)$. Recall, the measure $\mu_{\nu_{d,\lambda}}$ is the spectral measure of the $d$-block Galton-Watson graph with
	offspring distribution $d$Poisson$(\lambda)$.  Throughout, the distribution of $\mu_{\nu_{d,\lambda}}$ is denoted by $\Gamma_{d,\lambda}$, that is,
	\[
	\Gamma_{d,\lambda}(t)=\mu_{\nu_{d,\lambda}}((-\infty, t]), \mbox{ for all } t\in \R.
	\]
	We refer to \cite[Prpositions 5, 6]{adhikari2023spectrumrandomsimplicialcomplexes} for properties of the distribution $\Gamma_{d,\lambda}$.
	\begin{thm}\label{spe-conv}
		Let $\lambda,k,r,n,p$ be as in Assumption \ref{assum} and $H_n\in H(n,k,p)$. Let $A_{H_n}^{w_r}$ denote the $r$-set adjacency matrix of $G_r(H_n)$ as defined in \eqref{eqn:A_H}. Then the limiting spectral measure of  $A_{H_n}^{w_r}$ is $\mu_{\nu_{d,\lambda}}$ with $d=\binom{k}{r}-1$ almost surely. In other words,
		$$
		\frac{1}{N}\sum_{i=1}^{N}\delta_{\lambda_i} \rightsquigarrow \mu_{\nu_{d,\lambda}}, \mbox{ as $n\to \infty$, almost surely}
		$$
		where $N=\binom{n}{r}$ and $\lambda_1,\ldots,\lambda_{N}$ are the eigenvalues of $A_{H_n}^{w_r}$. 
		
	\end{thm}
	In the last result, we assume that the expected degree of each vertex is $\lambda$, a fixed constant. 
	In this paper, we do not attempt to find the LSD of $A^{w_r}_H$ when the expected degree of each vertex goes to infinity, that is, ${n-r \choose k-r} p = \lambda$ as $n \to \infty$. In this regime, it has been explored in \cite{SSMukherjee2024} (for the case of 1-set adjacency matrices, with arguments that extend to general $r$-sets) that the limiting spectral distribution is semicircular, along with further results such as the convergence of edge eigenvalues. See also \cite{vengerovsky2025eigenvalue} for related developments in the sparse case.
	
	For $k=2$ and $r=1$, the adjacency matrix $A_{H}^{w_r}$ is same as the adjacency matrix of $\mathcal G(n,p)$. The LSD of the adjacency matrix of $\mathcal G(n,p)$ has been studied in \cite{furedi1981eigenvalues,knowles2017eigenvalue} and in \cite{bordenave2017mean} when $np=\lambda$. We have the following corollary of Theorem \ref{spe-conv}.
	\begin{cor}\label{spec-ER}
		Let $\lambda>0$ and $G_n\in \mathcal G(n,p)$, the Erd\H os-R\' enyi graphs, with $np\to \lambda$ as $n\to \infty$. The LSD of the adjacency matrix of $G_n$ is ${\Gamma_{1,\lambda}}$ almost surely.
	\end{cor}
	For $r=1$, the vertices are serving as roots. The adjacency matrix $A^{w_1}_{H}$ is generalization of some notion of conventional adjacency matrices associated with hypergraphs \cite{Banerjee-hypmat,samiron-23}.  An immediate Corollary of Theorem \ref{spe-conv} is the following.
	\begin{cor}\label{spec-r=1}
		Let $\lambda>0$ and $H_n\in H(n,k,p)$. Then the LSD of $1$-set adjacency matrix $A_{H_n}^{w_1}$  of $H_n$ is $\Gamma_{k-1,\lambda}$ almost surely.
	\end{cor}
	We have already mentioned that the $k$-set adjacency matrix of the $(k+1)$-uniform random hypergraphs is the same as the adjacency matrix of the Linial Meshulam model  of dimension $k$,
	as appeared in \cite{adhikari2023spectrumrandomsimplicialcomplexes}. Thus, we have the following corollary, which has been established in \cite{adhikari2023spectrumrandomsimplicialcomplexes}.
	\begin{cor}\label{spec-r=k-1}
		Let $\lambda>0$ and $Y_k(n,p)$ be the Linial-Meshulam complexes of dimension $k$ on $n$ vertices with $np \to \lambda$ as $n\to \infty$. Then the LSD of the adjacency matrices of $Y_k(n,p)$ is $\Gamma_{k,\lambda}$ almost surely.
	\end{cor}

	\section{Convergence of $1$-set rooted uniform hypergraphs }\label{se:r=1} In this section, we provide a proof of Theorem \ref{deviationto0}. Although the result follows directly from Theorem \ref{main}, it represents the simplest case of Benjamini-Schramm convergence for $k$-uniform hypergraphs. For this reason, we present an independent proof. We believe that proof will motivate the proof of the Theorem \ref{main} and enhance the readability of the paper.
	The following two propositions are the key steps to prove the result.
	
	\begin{prop}[Convergence in expectation]\label{conv-exp-1}
		Let $\lambda,k,r,n,p$ be as in Assumption \ref{assum} with $r=1$ and $H_n\in H(n,k,p)$. Then the sequence of expected measures ${{\E}[U_1(H_n)]}$ converges weakly to the measure $\nu_{k-1,\lambda}$. 
		That is,
		$$
		\E[U_1(H_n)]\rightsquigarrow \nu_{k-1,\lambda}, \mbox{ as } n\to \infty.
		$$
	\end{prop}
	\begin{prop}[Almost sure convergence]\label{almost-sure-1} 	Let $\lambda,k,r,n,p$ be as in Assumption \ref{assum} with $r=1$ and $H_n\in H(n,k,p)$. Then the following holds
		$$
		\lim\limits_{n\to\infty}(U_1(H_n)-\E[U_1(H_n)])=0, \mbox{ almost surely}.
		$$
	\end{prop}
	
	\begin{proof}[Proof of  Theorem \ref{deviationto0}]
		The result follows from Proposition \ref{conv-exp-1} and  Proposition \ref{almost-sure-1}.
	\end{proof}
	
	\noindent The rest of the section is dedicated to prove these propositions.
	Before that, we recall the Portmanteau Theorem \cite[ Theorem 3.2]{bordenave2012lecture}. In \eqref{weak}, we have described weak convergence. The Portmanteau theorem provides equivalent criteria for weak convergence of probability measures. We use these equivalent criteria in proofs.
	\begin{thm}[Portmanteau Theorem]\label{Portmanteau}
		Let $(S, d)$ be a metric space, and let $\{P_n\}$ and $P$ be Borel probability measures on $S$. Then the following statements are equivalent:
		
		\begin{enumerate}[label=(\roman*)]
			\item $P_n \Rightarrow P$, i.e., $P_n$ converges weakly to $P$.
			
			\item For every bounded, uniformly continuous function $f : S \to \mathbb{R}$,
			\[
			\int f \, dP_n \to \int f \, dP.
			\]
			
			\item For every closed set $F \subseteq S$,
			\[
			\limsup_{n \to \infty} P_n(F) \leq P(F).
			\]
			
			\item For every open set $G \subseteq S$,
			\[
			\liminf_{n \to \infty} P_n(G) \geq P(G).
			\]
			
			\item For every Borel set $A \subseteq S$ such that $P(\partial A) = 0$ (i.e., $A$ is a continuity set for $P$),
			\[
			\lim_{n \to \infty} P_n(A) = P(A).
			\]
		\end{enumerate}
	\end{thm}
	\subsection{Convergence in expectation}\label{conv-exp-1-pr}We now prove Proposition \ref{conv-exp-1}, adapting ideas from \cite[Chapter 3]{bordenave2012lecture}. To this end, we begin with the following proposition.
	For any weighted rooted graph $g$,  the $\delta$-neighbourhood of the $1$-set rooted weighted-line graph 
	is given by
	$$
	B_\delta(g)=\{[G]\in \mathbb{G}^*_w:(G,o)_{\lceil\frac{1}{\delta}\rceil}\cong (g)_{\lceil\frac{1}{\delta}\rceil}\}.
	$$
	\begin{prop}\label{basis-open-convergence}
		Let $\lambda,k,r,n,p$ be as in Assumption \ref{assum} and $\mu_n=\E[U_1(H_n)]$, where $H_n\in H(n,k,p)$. Thus, for any $\delta>0$, 
		\[
		\lim\limits_{n\to\infty}\mu_n(B_\delta(g))=\nu_{k-1,\lambda}(B_\delta(g)),
		\]
		where $\nu_{k-1,\lambda}$ is the induced measure of $(k-1)$-block Galton-Watson measure.
	\end{prop}
	
	\begin{proof}[Proof of Proposition \ref{conv-exp-1}]
		In a metrizable topological space, a Borel probability measure is inner regular. By inner regularity of the measure $\nu_{k-1,\lambda}$, for $\eps>0$, there exists a compact subset $K\subseteq \mathbb{G}_w^*$ such that
		$$\nu_{k-1,\lambda}(K)\ge\nu_{k-1,\lambda}(\mathbb{G}_w^*)-\epsilon=1-\epsilon.$$
		Since $K$ is compact, for $\delta>0$, there exists finite collection of rooted graphs ${S} ( \subset \mathbb{G}^*_w)$ such that $K\subseteq \bigcup\limits_{g\in  {S} }B_{\delta}(g)$. Thus, there exists $\delta>0$ such that
		\begin{align}\label{eqn:lower1}
			\sum\limits_{g\in {S} }\nu_{k-1,\lambda}(B_{\delta}(g))\ge1-\epsilon.
		\end{align}
		Therefore, using Proposition \ref{basis-open-convergence}, there exists $N_o\in \N$ such that 
		\begin{align}\label{eqn:lower2}
			\sum\limits_{g\in  {S} }\mu_n(B_{\delta}(g))\ge1-2\epsilon, \mbox{ for all } n\ge N_0.
		\end{align}
		Let $f:\mathbb{G}_w^*\to\mathbb{R}$ be a uniformly continuous bounded function. 
		Then
		\begin{align*}
			&	\left|\int fd\mu_n-\int fd\nu_{k-1,\lambda}\right|
			\\\le& \left|\int_K fd\mu_n-\int_K fd\nu_{k-1,\lambda}\right|+\left|\int_{K^c} fd\mu_n-\int_{K^c} fd\nu_{k-1,\lambda}\right|
			\\\le&\sum\limits_{g\in {S}}\|f\|_\infty|(\mu_n(B_{\delta}(g))-\nu_{k-1,\lambda}(B_{\delta}(g)))|
			+\int\limits_{\mathbb{G}_w^*\setminus \mathcal{S}}\|f\|_\infty(d\mu_n-d\nu_{k-1,\lambda}),
		\end{align*}
		where $\mathcal{S}=\bigcup\limits_{g\in  {S} }B_{\delta}(g)$. Observe that, by \eqref{eqn:lower1} and \eqref{eqn:lower2}, we have
		\[
		\int\limits_{\mathbb{G}_w^*\setminus \mathcal{S}}\|f\|_\infty(d\mu_n-d\nu_{k-1,\lambda})\le\|f\|_\infty3\epsilon.
		\]
		Therefore, as $\|f\|_\infty$ is finite and $\epsilon$ is arbitrary,   Proposition \ref{basis-open-convergence} implies that 
		\[
		\lim\limits_{n\to\infty} \left|\int fd\mu_n-\int fd\nu_{k-1,\lambda}\right|=0.
		\]
		Hence the result, that is,  $\E[U_1(H_n)]\rightsquigarrow \nu_{k-1,\lambda}$, as $n\to \infty$.
	\end{proof} 
	
	It remains to prove Proposition \ref{basis-open-convergence}. The next subsection is dedicated for that.
	
	\subsubsection{\bf Proof of Proposition \ref{basis-open-convergence}}
	For the proof, we need to analyze the deviation
	\begin{align}\label{eqn:difference}
		\left|\mu_n(B_\delta(g))-\nu_{k-1,\lambda}(B_\delta(g))\right|.
	\end{align}  
	The following Lemma would help  us to understand the meaning of $\mu_n(B_\delta((G_1(H_n),i)))$ for some $H_n\in  H(n,k,p)$.
	\begin{lem}\label{EU}
		Recall $\mathcal{B}_{m}$ denotes the Borel $\sigma$-algebra on  $\mathbb{G}_w^*$ with respect to the metric $m$. For any Borel set $A\in \mathcal{B}_m$, 
		$$
		\mu_n(A)=\P(((G_1(H_n))(1),1)\in A),
		$$
		where $\mu_n=\E[U_1(H_n)]$. In particular, 
		for any $g=(G,w,u)\in\mathbb{G}^*_w$, 
		\begin{align}\label{eqn:mun}
			\mu_n(B_\epsilon(g))=\P((G_1(H_n(1)),1))_{\lceil\frac{1}{\epsilon}\rceil}\cong g_{\lceil\frac{1}{\epsilon}\rceil})
		\end{align}
		for any root $i\in [n]$.
	\end{lem}

	\begin{proof} Observe that we have
		\begin{align*}
			\mu_n(A)&=\frac{1}{n}\sum\limits_{i\in[n]}\E(\delta_{(G_1(H_n(i)),i)})(A)\\
			&=\frac{1}{n}\sum\limits_{i\in[n]}\P((G_1(H_n(i)),i)\in A)\\
			&=\P((G_1(H_n(1)),1)\in A).
		\end{align*}
		The last equality follows from the structure of $H_n\in H(n,k,p)$, all edges appear with equal probability. 
	\end{proof}

	Thus \eqref{eqn:mun} gives a description of $\mu_n(B_\delta(g))$. It is noteworthy to know the meaning of  $\nu_{k-1,\lambda}(B_\delta(g))$.  Observe that, for any $g=(G,w,u)\in\mathbb{G}^*_w$, we have
	\begin{align}\label{eqn:nun}
		\nu_{k-1,\lambda}(B_\delta(g))=\P(G_1(GW_{k-1},o)_{\lceil\frac{1}{\delta}\rceil}\cong g_{\lceil\frac{1}{\delta}\rceil}),
	\end{align}
	where $G_1(GW_{k-1},o)$ denotes the 1-set line graph of $(GW_{k-1},o)_{\lceil\frac{1}{\epsilon}\rceil}$.  
	Thus Equation \eqref{eqn:mun} and \eqref{eqn:nun} imply that the difference noted in \eqref{eqn:difference} primarily arises from discrepancies in the neighbourhoods of $(G_1(H_n)(1),1)$ and $(GW_{k-1},o)$. To analyze this more precisely, we now investigate the structure of the rooted random hypergraph $(H_n, 1)$, where $H_n\in H(n,k,p)$, which leads to the introduction of certain random variables that measure the extent of deviation between the respective neighbourhoods.
	\subsubsection{An exploration process in a rooted hypergraph}\label{EXPL} Let $H_n\in H(n,k,p)$. By Lemma \ref{EU}, it is enough to know the distribution of $(G_1(H_n), 1)$.
	We introduce a breadth-first search approach for a finite rooted hypergraph $(H_n, 1)$ to analyze the deviation of $(G_1(H_n), 1)$ from the structure of $(GW_{k-1},o)$. This iterative procedure constructs a bijection $\phi:\mathcal S \to V(H_n(1))$, where $\mathcal S \subset \mathbb{N}_k^\infty$ and $V(H_n(1))$ denotes the vertex set of the connected component of $H_n$ containing the root vertex $i$. Given a vertex $j \in V(H_n(1))$, we declare $j$ explored once $\phi^{-1}(j)$ is defined.
	The bijection $\phi$ is defined iteratively based on the ordering of $\mathbb{N}_k^\infty$ described earlier.
	We start by setting $\phi(o)=1$, the root; $A_0=\{1\}$, the initial active set; $C_0=\emptyset$, the initial covered component; $U_0=V(H)\setminus\{1\}$, the initial unexplored component; $CE_0=\emptyset$, the initial covered hyperedge set; $UE_0=E(H)$, the initial unexplored hyperedge set; and proceed in the following iteration:
	\begin{enumerate}
		\item {\it Active vertices:} $A_{t+1}=A_t\setminus\{v_{t+1}\}\cup I_{t+1}$.
		\item {\it Unexplored vertices:} $U_{t+1}=U_t\setminus I_{t+1}$. 
		\item {\it Covered component:} $ C_{t+1}=C_t\cup\{v_{t+1
		}\}$. 
		\item {\it Covered edges:} $CE_{t+1}=CE_{t}\cup R_{t+1}$. 
		\item {\it Unexplored edges:} $UE_{t+1}=UE_t\setminus R_{t+1}$. 
	\end{enumerate} Where $I_t$ and $R_t$ are defined in the following discussion.
	We define $v_{t+1}=\phi(\mathbf{i}_t)$, where $\mathbf{i}_t=\min\{\phi^{-1}(v):v\in A_t\}$, the minimum is with respect to the ordering on $\mathbb{N}^\infty_{\binom{n}{r}}$. The collection of \emph{unexplored hyperedges} containing $v_{t+1}$ is denoted by $$R_{t+1}=E_{v_{t+1}}(H)\cap UE_t.$$  Suppose that $\beta(t+1)=|R_{t+1}|$. If $\beta(t+1)>0$, we enumerate $R_{t+1}$ as 
	\begin{align}\label{eqn:Rt}
		R_{t+1}=\{e_{t+1,1},e_{t+1,2},\ldots,e_{t+1,\beta(t+1)}\}.
	\end{align}
	The set of \emph{unexplored neighbours} of the vertex $v_{t+1}$ is given by 
	\begin{equation}\label{I-t-r=1}
		I_{t+1} = \left( \bigcup\limits_{e \in R_{t+1}} e \right) \cap U_t.
	\end{equation}
	We set $	X_{t+1}=|I_{t+1}|$  for $t\in \N\cup \{0\}$. In that case by Equation \ref{I-t-r=1}
	\begin{equation}\label{eqn-X-t-1}
		X_{t+1}\le (k-1)|R_{t+1}|.
	\end{equation}
	\paragraph*{Iteration of the bijection $\phi$} To describe the iteration of $\phi$, we need to enumerate the vertex set of $H_n$ using the breadth-first-search.
	\begin{nt}
		The hyperedge set of the connected component of $H_n$ that contains the root can be expressed as disjoint union of $R_t$.
	\end{nt}
	\begin{proof}
		For two distinct $t_1,t_2\ge 1$, without loss of generality suppose that $t_1<t_2$. In that case, we have   $R_{t_1}\cap UE_{t_1}=\emptyset$ and $UE_{t_2}\subseteq UE_{t_1} $. Therefore the sets $R_{t_1}$ and $R_{t_2}$ are mutually disjoint set. Thus, if $a = |V(H)(1)|$ is the number of vertices in the connected component containing vertex $1$ then 
		\[
		E(H(1)) = \bigsqcup\limits_{t=1}^{b} R_{t}
		\]
		for some number $b \le a$.
	\end{proof}
	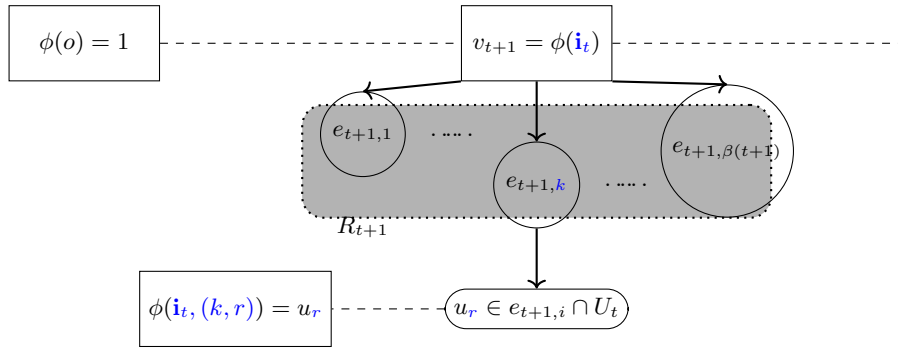
\begin{figure}[b]
		\centering
		\begin{tikzpicture}[
			node distance=1.2cm and 1.8cm,
			every node/.style={font=\small},
			box/.style={draw, minimum width=2cm, minimum height=1cm, align=center},
			ellipse box/.style={draw, rounded corners=8pt, minimum width=5.5cm, minimum height=2.2cm},
			curved arrow/.style={->, draw, thick, bend left=20},
			curved arrowr/.style={->, draw, thick, bend right=20}
			]
			
			\node[box] (0) {$\phi(o) = 1$};
			
			\node[box, right=4cm of 0] (utplus1) {$v_{t+1} = \phi({\color{blue}\mathbf{i}_t})$};
			
			\node[draw, text width=6cm,minimum height=1.5cm, align=center, rounded corners=8pt,fill=black!30!white, below=0.3 cm of utplus1,dotted,thick] (rect) {};
			
			\node[draw,circle,below left=0.3cm and 0.9cm of utplus1] (e1) {$e_{t+1,1}$};
			\draw[->, thick] (utplus1.south west) -- (e1.north);
			
			\node[draw,circle,below=0.8cm of utplus1] (e2) {$e_{t+1,{\color{blue}k}}$};
			\draw[->, thick] (utplus1) -- (e2);
			\node[right=0.2 cm of e1]{$\ldots$};
			\node[right=0.4 cm of e1]{$\ldots$};
			\node[right=0.2 cm of e2]{$\ldots$};
			\node[right=0.4 cm of e2]{$\ldots$};
			\node[draw,circle,below right=0.3cm and 0.9cm of utplus1] (ei) {$e_{t+1,\beta(t+1)}$};
			\draw[->, thick] (utplus1.south east) -- (ei.north);
			\node[ right=4cm of utplus1] (ej) {};
			\node[draw, align=center, rounded corners=8pt,below=0.8cm of e2] (intersect) {$u_{{\color{blue}r}}\in e_{t+1,i} \cap U_t$};
			\draw[->, thick] (e2) -- (intersect);
			\draw[dashed] (utplus1) -- (ej);
			
			\node[box, left=1.5cm of intersect] (phi1) {$\phi({\color{blue}\mathbf{i}_t, (k,r)}) = u_{\color{blue}r}$};
			\draw[dashed](0) to (utplus1);
			\draw[dashed] (intersect) to (phi1);
			
			\node[below =0.4cm and 4.5cm of e1] (Rtplus1) {$R_{t+1}$};
			
		\end{tikzpicture}
		\caption{Iterative construction of the map $\phi$.}
		\label{fig:const-phi}
	\end{figure}
	
	\begin{nt}
		Enumeration of $R_t$ provide us the iterative definition of $\phi$.
	\end{nt}
	\begin{proof}
		Enumerating each $R_t$ induces an ordering on the edges of $H(1)$: for any two distinct edges $e_{t_1,j}, e_{t_2,j'} \in E(H(1))$, we write $e_{t_1,j'} < e_{t_2,j}$ if 
		\[
		\mbox{either (i) $t_1 < t_2$, or (ii) $t_1 = t_2$ and $j'< j$.}
		\]
		
		If $X_{t+1}>0$ then for any $u\in I_{t+1}$, suppose that $k=\min\{j:u\in e_{t+1,j}\in R_{t+1}\}$. We enumerate $U_t\cap e_{t+1,k}=\{u_1,u_2,\ldots\}$, and  if $u=u_{r}$, then we define $\phi(\mathbf{i}_t, (k,r))=u$, see Figure \ref{fig:const-phi}.  
	\end{proof}    
	\noindent Let $t_0$ denote the time when the process terminated, that is, $A_{t_0}=\emptyset$ but $A_{t_0-1}\ne \emptyset$. For $t<t_0$, the knowledge gained at breadth-first exploration leads us to the filtration at $t$-th time step is represented by the \emph{filtration}  $$\mathcal{F}_t=\sigma((A_0,C_0,U_0,CE_0,UE_0),(A_1,C_1,U_1,CE_1,UE_1),\ldots,(A_t,C_t,U_t,CE_t,UE_t)).$$
	This filtration will be used later in various proofs.
	\begin{nt}
		For all $t\le t_0$,
		\begin{align}\label{eqn:sumrelation}
			n=|A_t|+t+|U_t|.
		\end{align}
	\end{nt}
	\begin{proof}
		Let $a=|V(H(1))|$. For any $ t<a$, at $t$-th step $v_{t+1}$ is removed from $A_t$ and new $X_{t+1}$ number of elements added in $A_t$ to form $A_{t+1}$.
		Therefore, $|A_{t+1}|=|A_t|-1+X_{t+1}$ for all $t\ge 0$. Which implies that, for $t\ge 0$,
		\begin{align}\label{eqn:A_n}
			|A_{t+1}|=1+\sum\limits_{s=1}^{t+1}(X_s-1).
		\end{align}
		The iteration ended when $A_t=\emptyset$. Then the vertex set $V(H)$ can be written as $V(H)=A_t\cup C_t\cup U_t$ for any $t\le t_0$. By definition of the iteration, exactly one new vertex is added to $C_t$ at each step. Thus, $|C_t|=t$. Therefore, if $H$ is a hypergraph on $n$ vertex, $n=|A_t|+t+|U_t|$.
	\end{proof}
	Since the exploration is going on a random hypergraph, the quantities $X_t$ and $|A_t|$ are random variables. Let $Z_t$ denote the number of hyperedges $e$ such  that $e\setminus\{v_t\}\subseteq U_t$. The random variable $Z_t$ is as follows, given $|U_t|$,
	\[
	\P(Z_{t}=i)=\binom{r_{t}}{i}p^i(1-p)^{r_{t}-i}, \mbox{ for all $i=0,1,\ldots,r_{t}$},
	\]
	where $r_{t}=\binom{|U_{t}|-1}{k-1}$.  We set $Y_1=|E_{i}(H)|$, the star of the root. For  $t> 0$, we denote 
	\begin{align}\label{eqn:Yn}
		Y_{t}=|E_{v_{t}}(H)|-1.
	\end{align}
	In other words  $Y_t=D_{v_t}-1$, recall $D_{v_t}$ denotes the degree of the vertex $v_t$.

	Note that $d$-block Galton -Watson process the children of any individual never belongs to the past generations. Therefore if $(H_n,1)\cong(T_1,o)$ then,   for any vertex $v_t$, each offspring hyperedge $e$ of $v_t$ satisfy the condition $e\setminus\{v_t\}\subseteq U_t $. Therefore, we always have $Y_t=Z_t$.  Thus, the difference $Y_t-Z_t$ is one of the indications of deviation of $(H_n,1)$ from $(T_1,o)$. The following remarks give a road map of the proof of Proposition \ref{basis-open-convergence}.

	\begin{rem}\label{discrimination}\rm The deviation in \eqref{eqn:difference} depends on the following factors:  
		\begin{enumerate}[leftmargin=*]
			\item The difference between the degree distribution in the 
			random hypergraph $H_n\in H(n,k,p)$ and the Poisson distribution with mean
			$\lambda$. Moreover, in a Galton-Watson $k$-hypertree, the degrees of
			different vertices are independent random variables. In contrast, in the
			random hypergraph $H_n$, the degree distributions of distinct 
			vertices exhibit correlation. This discrimination is quantified in Lemma \ref{fact-1-2}.
			\item How acyclic the structure of $H_n$ is. Here, we should clarify what the acyclic structure of $H_n$ means. If a hypergraph has a cycle, then at some stage of the breadth-first exploration, we have two distinct vertices with the same offspring vertices. For instance, if we explore the graph illustrated in Figure \ref{fig:graph-cycle} starting from the vertex $1$, then $2$ and $5$ have the same offspring $4$. This will not be the case if we explore a tree in place of the graph. Similarly, from Section \ref{GLTH}, it is evident that for two distinct hyperedges $e$ and $e'$ in a $k$-uniform rooted hypertree obtained in the branching process described in Section \ref{GLTH} with $r=1$, then $|e\cap e'|\le1$ with the following two possibilities:
			\begin{enumerate}
				\item If $e$ and $e'$ are the offspring hyperedges of the  vertex $v$, then $e\cap e'=\{v\}$.
				\item If $e$ and $e'$ are not offspring of the same vertex, then $e\cap e'=\{v\}$ implies that one of $e$ and $e'$ is offspring hyperedge of $v$ and the other is the parent hyperedge of $v$.
			\end{enumerate}
			For example, the hypergraph illustrated in the Figure \ref{fig:hypwithcycle}, 
			if one starts exploring it from vertex $1$, then at the vertex $3$ we obtained two offspring hyperedges $\{3,5,6\}$ and $\{3,6,7\}$ of the vertex $3$ intersect at $2$ vertices. Furthermore, at vertex $5$, it is observed that $5\in \{3,5,6\}\cap\{5,6,8\}$ with $\{5,6,8\}$ is an offspring hyperedge of $5$ but the hyperedge $\{3,5,6\}$ is not a parent hyperedge of $5$. This contrast related to cyclic structure is formally captured in Lemma \ref{hypertree-1} and Lemma \ref{hypertree-2}.
		\end{enumerate}
	\end{rem}

	\noindent Our next result demonstrates that, due to the condition $\binom{n-1}{k-1}p = \lambda$, the first two of the three factors mentioned above occur with very small probability. However, before proving this result, we first establish the following inequality, which will be useful in the proof.
	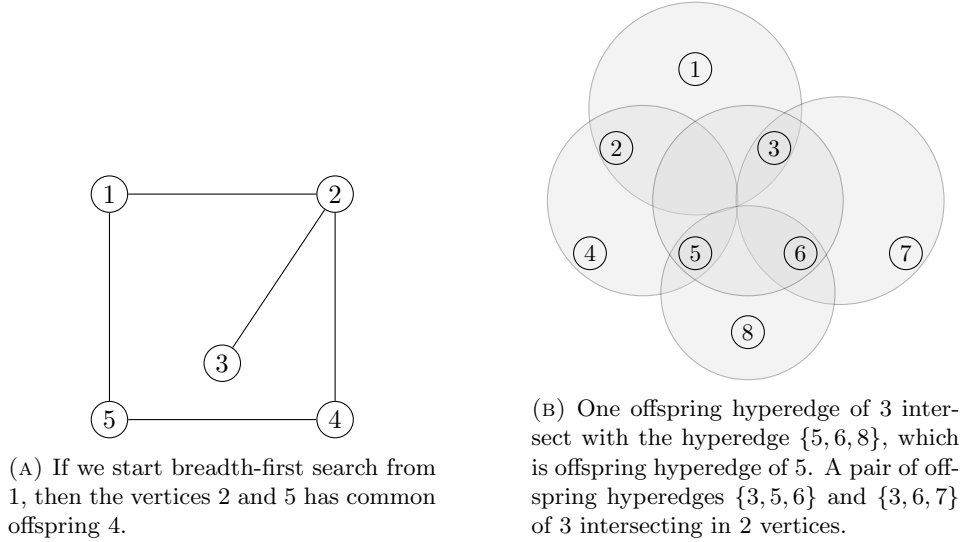
\begin{figure}[t]
		\centering
		\begin{subfigure}[b]{0.45\linewidth}
			\centering
			\begin{tikzpicture}[scale=1.5, every node/.style={draw, circle, inner sep=2pt}]
				\node (1) at (0,2) {1};
				\node (2) at (2,2) {2};
				\node (3) at (1,0.5) {3};
				\node (4) at (2,0) {4};
				\node (5) at (0,0) {5};
				
				\draw (1) -- (2);
				\draw (2) -- (4);
				\draw (4) -- (5);
				\draw (5) -- (1);
				\draw (2)--(3);
				
			\end{tikzpicture}
			\caption{If we start breadth-first search from $1$, then the vertices $2$ and $5$ has common offspring $4$.}
			\label{fig:graph-cycle}
		\end{subfigure}
		\hfill
		\begin{subfigure}[b]{0.45\linewidth}
			\centering
			\begin{tikzpicture}[scale=0.7, every node/.style={draw, scale=0.9,circle, inner sep=2pt},
				hyperedge/.style={draw, fill=gray!30, opacity=0.3, rounded corners, inner sep=2pt}]
				
				\node (1) at (4,6) {1};  
				\node (2) at (2.5,4.5) {2};
				\node (3) at (5.5,4.5) {3};
				\node (4) at (2,2.5) {4};
				\node (5) at (4,2.5) {5};
				\node (6) at (6,2.5) {6};
				\node (7) at (8,2.5) {7};
				\node (8) at (5,1) {8};  
				
				\begin{scope}[on background layer]
					\node[hyperedge, fit=(1) (2) (3)] {};  
					\node[hyperedge, fit=(2) (4) (5)] {};  
					\node[hyperedge, fit=(3) (6) (7)] {};  
					\node[hyperedge, fit=(5) (6) (8)] {};  
					\node[hyperedge, fit=(5) (6) (3)]{};  
				\end{scope}
			\end{tikzpicture}
			\caption{One offspring hyperedge of $3$ intersect with the hyperedge $\{5,6,8\}$, which is offspring hyperedge of $5$. A pair of offspring hyperedges $\{3,5,6\}$ and $\{3,6,7\}$ of $3$ intersecting in $2$ vertices.  }
			\label{fig:hypwithcycle}
		\end{subfigure}
		\caption{Cyclic structures in uniform-hypergraphs.}
		\label{fig:mainfig}
	\end{figure}
	
	\begin{lem}\label{lem-inequality}
		For $0<x<1$, the following holds
		$$ 1-\left( (1-x)^{k_1}+(1-x)^{k_2}+\ldots+(1-x)^{k_n}\right)\le (k_1+k_2+\ldots+k_n)x.$$
	\end{lem}
	\begin{proof}
		Suppose that $f(x)= 1-\left( (1-x)^{k_1}+(1-x)^{k_2}+\ldots+(1-x)^{k_n}\right)$ and $g(x)=(k_1+k_2+\ldots+k_n)x$.
		Since $0<x<1$, it holds that $$f'(x)=k_1(1-x)^{k_1-1}+\ldots+k_n(1-x)^{k_n-1}\le k_1+k_2+\ldots+k_n=g'(x).$$
		Moreover, $f'(0)=g'(0)$. Therefore, for $0<x<1$,
		\[
		f(x)-g(x)\le f(0)-g(0)\le 1-n\le 0.
		\]
		This completes the proof.
	\end{proof}
	\begin{lem}\label{lem-a>b ineq}
		For two positive  real numbers $a,b$ with $ b\le a$,
		$$
		a^n-b^n\le a^{n-1}(a-b).
		$$
	\end{lem}
	\begin{proof}
		We expand the difference of powers as  
		$$
		a^{n} - b^{n} \;=\; (a-b)\sum_{k=0}^{n-1} a^{\,n-1-k} b^{\,k}.
		$$  
		Because $0\le b\le a$, each summand is non-negative, and in particular,  
		$$
		\sum_{k=0}^{n-1} a^{\,n-1-k} b^{\,k} \;\le\; a^{\,n-1}.
		$$  
		Multiplying through by $a-b>0$ then yields  
		$$
		a^{n} - b^{n} \;\le\; a^{\,n-1}(a-b).
		$$  
		Hence the result.
	\end{proof}
	Suppose that $(\Omega,\mathcal{F})$ is a measure space and $\P_1$ and $\P_2$ are two probability measures defined on the same measure space. The \emph{total variational distance} between $\P_1$, and $\P_2$ is defined as 
	\[
	d_{TV}(\P_1,\P_2)=\sup\limits_{A\in\mathcal{F}}|\P_1(A)-\P_2(A)|.
	\]
	Given two Random variable $X$, and $Y$ with the same sample space and distribution $\P_X$, and $\P_Y$, respectively, the \emph{total variational distance} between $X$ and $Y$ is $d_{TV}(X,Y):=d_{TV}(\P_X,\P_Y)$. The total variation distance (TVD) between a Binomial distribution Binomial$(n,p)$ and a Poisson distribution Poisson $(\lambda)$ (where $np=\lambda$) is given by \cite[Equation-2.7]{bordenave2012lecture}, 
	\begin{equation}\label{tv-bin-poi}
		d_{\text{TV}}\big(\text{Binomial}(n, p), \text{Poisson}(\lambda)\big) \le \frac{\lambda}{n}.
	\end{equation}
	Recall $D_i$ denotes the degree of the $i$-th vertex in the random uniform hypergraphs $H_n\in H(n,k,p)$. Observe that, from the discussion in Section \ref{exp-degree}, $D_{i}\sim \text{Binomial}(r,p)$, where $r=\binom{n-1}{k-1}$. Therefore \eqref{tv-bin-poi} leads us to the following equation:
	\begin{equation}\label{D_i-poi}
		d_{\text{TV}}\big(D_i, \text{Poisson}(\lambda)\big) \le \frac{\lambda}{\binom{n-1}{k-1}},\text{~where~}\lambda=\binom{n-1}{k-1}p. 
	\end{equation}
	For large values of $ n$, the degree distribution behaves like the Poisson distribution in sparse random graphs and hypergraphs.
	This limiting behaviour of graphs reveals a fascinating asymptotic property of random graphs. 
	
	\begin{lem}\label{fact-1-2}
		Let $\lambda,k,r,n,p$ be as in Assumption \ref{assum}. There exists $Z'_1,\ldots,Z'_{t_0}$ be i.i.d.  $\text{Poission}(\lambda)$ random variables such that, for sufficiently large $n$, 
		\begin{equation*}
			\P((Z'_1,\ldots,Z'_{t_0})\ne (Y_1,\ldots,Y_{t_0}))\le( t_0+\lambda(k-1) t_0^2)\frac{c}{(n-1)}+\frac{\lambda t_0}{(n-1)^{k-1}},
		\end{equation*}
		where $\{Y_t\}$ is as described in the exploration process \eqref{eqn:Yn}.
	\end{lem}
	\begin{proof}
		The equality $\{Z_t=Y_t\}$ of random variables represents the event in which, except one hyperedge, the vertex $v_{t+1}$ forms no other hyperedge with vertices outside of $U_t$. Therefore, by using Lemma \ref{lem-inequality}, we have
		\begin{align*}
			&\P(\{Z_t\ne Y_t\}\given \mathcal F_t)
			\\=&1-\left((1-p)^{\binom{n-|U_t|}{k-1}}+(1-p)^{\binom{n-|U_t|}{k-2}\binom{|u_t|-1}{1}}+\ldots +(1-p)^{\binom{n-|U_t|}{1}\binom{|u_t|-1}{k-2}}\right)
			\\\le
			& \sum\limits_{j=0}^{k-2}\binom{n-|U_t|}{k-1-j}\binom{|U_t|-1}{j}p
			\\=
			& \l(\binom{n-1}{k-1}-\binom{|U_t|-1}{k-1
			}\r)\frac{\lambda}{\binom{n-1}{k-1}}.
		\end{align*}
		We write  $a_n\approx b_n$ for  large $n$  if $a_n/b_n\to 0$ as $n\to \infty$.
		As for  large $n$, it holds that $ (\binom{n-1}{k-1}-\binom{|U_t|-1}{k-1
		})\approx  \left((n-1)^{k-1}-(|U_t|-1)^{k-1}\right)\frac{1}{(k-1)!}$. Therefore, for  large $n$, using Lemma \ref{lem-a>b ineq} we have the following:
		\begin{align*}
			\P(\{Z_t\ne Y_t\}\given \mathcal F_t)&
			\le \left((n-1)^{k-1}-(|U_t|-1)^{k-1}\right)\frac{\lambda}{(n-1)^{k-1}}\frac{1}{(k-1)!}\\
			&\le (n-|U_t|)\frac{c}{(n-1)}= (|A_t|+t)\frac{c}{(n-1)},
		\end{align*}
		where $c=\lambda/(k-2)!$. The last equality follows from \eqref{eqn:sumrelation}. Again by \eqref{eqn-X-t-1} and \eqref{eqn:A_n}, we have
		\[
		\E(|A_t|)=	\E(|A_t||\mathcal{F}_t)=1+\sum\limits_{s=1}^t(\E(X_t|\mathcal{F}_t)-1)\le 1+\sum\limits_{s=1}^t(\lambda(k-1)-1)=1+(\lambda(k-1)-1)t.
		\] 
		Which implies that 
		\begin{align}\label{eqn:noteq}
			\P(Z_t\ne Y_t)\le (1+\lambda(k-1) t)\frac{c}{(n-1)}.
		\end{align}
		Now by \eqref{D_i-poi}, for sufficiently large $n$, there exists  a positive constant $c_0$ such that 
		\begin{align}\label{eqn:Tv}
			d_{TV}(Z_i,Z_i')\le \frac{c_0}{(n-1)^{k-1}}\text{~for all~}i=1,\ldots,t_0.
		\end{align}
		Since,	for  $t=1,2,\ldots,t_0$, we have $ |\P(Z_t\ne Y_t)-\P(Z'_t\ne Y_t)|\le d_{TV}(Z_t,Z_t')$. Hence, applying the union bound with \eqref{eqn:noteq} and \eqref{eqn:Tv}, we get
		\begin{align*}
			\P((Z'_1,\ldots,Z'_{t_0})\ne (Y_1,\ldots,Y_{t_0}))&\le \sum\limits_{t=1}^{t_0}\P(Z'_t\ne Y_t)
			\le \sum\limits_{t=1}^{t_0}[\P(Z_t\ne Y_t)+d_{TV}(Z_t,Z_t')]
			\\&\le ( t_0+\lambda (k-1) t_0^2)\frac{c}{(n-1)}+\frac{\lambda t_0}{(n-1)^{k-1}}.
		\end{align*} 
		This completes the proof.  
	\end{proof}
	\begin{lem}\label{hypertree-1} Let $\lambda,k,r,n,p$ be as in Assumption \ref{assum}, and $\{R_t\}$ be as defined in the exploration process \eqref{eqn:Rt}. Then, for  where $0\le t< \tau$, 
		\[
		\P
		(\{\exists e\in R_{t+1} \suchthat |e\setminus U_t|>1\})\le \frac{c_1t}{n-1},
		\]
		for some positive constant $c_1$.
	\end{lem}
	\begin{proof}
		Note that since $e\in R_{t+1}$ we have $v_{t+1}\in e\setminus U_t$. Thus, we have $|e\setminus U_t|\ge 1$. The equality $|e\setminus U_t|= 1$ holds if other than $v_{t+1}$ the remaining $(k-1)$ vertices are chosen from $U_t$ only.  Thus, given $|U_t|$, then 
		$|e\backslash U_t|> 1$ can occurs in
		\begin{align*}
			&\left(\binom{n-1}{k-1}-\binom{|U_t|}{k-1}\right)\text{~no of ways.}
		\end{align*}
		Therefore,  by the union bound, we have 
		\[
		\P
		(\{\exists e\in R_{t+1}\suchthat |e\setminus U_t|>1\}\given \mathcal F_t)\le \left(\binom{n-1}{k-1}-\binom{|U_t|}{k-1}\right)p.
		\]
		Again for  large $n$, we have $ \left(\binom{n-1}{k-1}-\binom{|U_t|}{k-1}\right)\approx ((n-1)^{k-1}-(|U_t|)^{k-1})\frac{1}{(k-1)!}$. By the similar calculation as in the proof of Lemma \ref{fact-1-2}, for some $c_1>0$, we have 
		\begin{align*}
			\P(\{\exists e\in R_{t+1}\suchthat |e\setminus U_t|>1\}\given \mathcal F_t) 
			&\le (t+|A_t|-1))\frac{c_1}{n-1}.
		\end{align*}
		Again, as shown in the proof of Lemma \ref{fact-1-2}, we have $\E(|A_t|)\le1+(\lambda-1)t$. Therefore 
		\[
		\P(\{\exists e\in R_{t+1}\suchthat |e\setminus U_t|>1\})\le \frac{c_1t}{n-1}.
		\]
		This completes the proof.
	\end{proof}
	
	The following lemma establishes that the probability of two distinct hyperedges, sharing the same parent vertex, intersecting at more than one vertex is small.
	\begin{lem}\label{hypertree-2}
		Let $\lambda,k,r,n,p$ be as in Assumption \ref{assum}, and $\{R_t\}$ be as defined in the exploration process \eqref{eqn:Rt}. Then, for $0\le t<\tau$, 
		\[
		\P
		(\{\exists~ e, e' \in R_{t+1} \suchthat e\neq e', |e\cap e'|>1\})\le \frac{ c_2 }{n-1},
		\]
		for some positive constant $c_2$ (depends on $k$, and $\lambda$).
	\end{lem}
	\begin{proof}
		Let $e,e'\in R_{t+1}$ such that  $v_{t+1}\in e\cap e'$. Therefore, by union bound,
		\begin{align*}
			\P(\{\exists e, e'\in R_{t+1}\suchthat |e\cap e'|>1\})&=\left(\binom{n-1}{k-1}\left(\binom{n-1}{k-1}-\binom{n-1-(k-1)}{k-1}\right)\right)p^2\\
			&=\left(\binom{n-1}{k-1}-\binom{n-k}{k-1}\right)\frac{\lambda^2}{(\binom{n-1}{k-1})}.
		\end{align*}
		For  large $n$, using Lemma \ref{lem-a>b ineq} we have
		\begin{align*}
			\P((\{\exists e, e'\in R_{t+1}\suchthat |e\cap e'|>1\})&\approx  \left((n-1)^{(k-1)}-(n-k)^{(k-1)}\right)\frac{\lambda^2  }{(n-1)^{(k-1)}}\\
			&\le(k-1)^2\frac{ \lambda }{n-1}.
		\end{align*}
		Thus, the result follows.
	\end{proof}
	\noindent Now we are ready to give a proof of Proposition \ref{conv-exp-1} using Lemma \ref{fact-1-2}, Lemma \ref{hypertree-1} and Lemma \ref{hypertree-2}.
	
	\begin{proof}[Proof of Proposition \ref{basis-open-convergence}]
		Let $g\in \mathcal G^*_w$. By \eqref{eqn:mun} and \eqref{eqn:nun} we have
		\begin{align}\label{eqn:difference1}
			&\left|\mu_n(B_\delta(g))-\nu_{k-1,\lambda}(B_\delta(g))\right|
			\\=&|\P((G_1(H_n(1)),w_{1,H},1))_{\lceil\frac{1}{\delta}\rceil}\cong g_{\lceil\frac{1}{\delta}\rceil})-\P(G_1(GW_{k-1},o)_{\lceil\frac{1}{\delta}\rceil}\cong g_{\lceil\frac{1}{\delta}\rceil})|\nonumber
		\end{align}
		Let $t_0=\lceil\frac{1}{\delta}\rceil$. We consider the following events
		\begin{align*}
			A_1&=\{(Z'_1,\ldots,Z'_{t_0\wedge t})\ne (Y_1,\ldots,Y_{t_0\wedge t})\}.
			\\A_2&=\{\exists 0\le s\le t_0\wedge t\suchthat |e\backslash U_t|>1 \mbox{ for some } e\in R_{t}\}.
			\\ A_3&=\{\exists 0\le s\le t_0\wedge t\suchthat |e\cap e'|>1 \mbox{ for some } e\neq  e'\in R_{t}\}.
		\end{align*}
		As discussed in Remark \ref{discrimination},  the deviation described by \eqref{eqn:difference1} depends on the following two reasons: Firstly, observe that $\P(G_1(GW_{k-1},o)_{\lceil\frac{1}{\delta}\rceil}\cong g_{\lceil\frac{1}{\delta}\rceil})\neq 0$ only if  $g= G_1(T)$ for some $(k-1)$-block tree $T$ of depth $t_0$ from its root. 
		The difference in \eqref{eqn:difference1} is non-zero because of the occurrence of event $A_1$.
		
		Secondly, for $d=k-1$, the $d$-block Galton-Watson tree $(GW_d,o)$ admits a natural representation of a $1$-set line graph of an $1$-set rooted $k$-uniform hypergraph $(T_k,o)$. The main characteristics of $(T_k,o)$ is for two hyperedges $e$ and $e'$, we must have  $|e\cap e'|\le 1$. The $|e\cap e'|=1$ occurs with the following two cases: (1)if $e$ and $e'$ are two offspring hyperedges of a vertex $v$ then $e\cap e'=\{v\}$, and (2) if $e$ and $e'$ are not the offspring of the same vertex then $e\cap e'=\{v\}$ implies one of $e$ and $e'$ is offspring hyperedge of $v$ and the other is the parent hyperedge of $e$. The deviation of \eqref{eqn:difference1} also depends on the fact how much $(G_1(H_n(1)),i)_{l}$  deviates from this characteristics of  $(T_k,o)$. In other words, the difference in \eqref{eqn:difference1} is non-zero if $A_2\cup A_3$ occurs.
		Thus, we have
		\[
		\left|\mu_n(B_\delta(g))-\nu_{k-1,\lambda}(B_\delta(g))\right|\le \P(A_1)+\P(A_2)+\P(A_3).
		\]
		Again, by the union bounds from Lemma \ref{hypertree-1} and Lemma \ref{hypertree-2} we have
		\begin{align*}
			\P(A_2)\le \frac{c_1(t_0\wedge t)^2}{n-1}
			\mbox{ and } \P(A_3)&\le \frac{c_2(t_0\wedge t)}{n-1}.
		\end{align*}	
		Therefore, using Lemma \ref{hypertree-1} and $t_0 \wedge t \le t_0$, we get  
		\begin{align*}
			\left|\mu_n(B_\delta(g))-\nu_{k-1,\lambda}(B_\delta(g))\right|	& \le \frac{c( t_0+\lambda t_0^2)}{(n-1)}+\frac{\lambda t_0}{(n-1)^{k-1}}
			+ \frac{c_1 t_0^2}{n-1}+\frac{c_2 t_0}{n-1}.
		\end{align*}
		Which tends to $0$ as $n\to \infty$. This completes the proof.
	\end{proof}
	

	\subsection{Almost sure convergence}\label{sec-Almost-sure-1} This subsection is dedicated to prove Proposition \ref{almost-sure-1}. The proof (which is partially motivated by \cite[Lemma 2.3 and Lemma 2.4]{dembo2010gibbs} and \cite[Section 8.4]{van2013random}), is divided in two steps. In the first step, we show that for large enough $n$, under the Assumption \ref{assum}, the number of hyperedges in $H_n\in H(n,k,p)$ is bounded by $\frac{3n\lambda}{2k}$ almost surely. See Lemma \ref{hyperedge bound}. Using this fact, in the second step, we construct a martingale sequence. This martingale leads to the desired result using Azuma Inequality \cite[Theorem 5.2]{chung2006concentration}. See Lemma \ref{c-lem}.
	
	Recall that  each hyperedge  $e\in \left[\binom{n}{k}\right]$ is included in $H_n\in H(n,k,p)$ with probability $p$ and independently. Let $\{X_e\suchthat e\in \left[\binom{n}{k}\right] \}$ be i.i.d.  Bernoulli($p$) random variables. Equivalently, we say an edge $e\in \left[\binom{n}{k}\right]$ is included in $H_n\in H(n,k,p)$ if and only if $X_e=1$. Thus, the total number of hyperedges in the random hypergraph $H_n\in H(n,k,p)$ is
	\begin{align}\label{eqn:xin}
		\xi_n=\sum\limits_{e\in \left[\binom{n}{k}\right]}X_e.
	\end{align}
	The next result gives an almost sure bound on  $\xi_n$. 
	
	\begin{lem}\label{hyperedge bound}
		Let $\lambda,k,r,n,p$ be as in Assumption \ref{assum}, and $\xi_n$ be as defined in \eqref{eqn:xin}. Then 
		$$
		\P\l(\l\{\xi_n>\frac{3n\lambda}{2k}\r\}\r)\le e^{-\frac{\lambda n}{6k}}.
		$$ 
	\end{lem}
	
	Next, we consider the lexicographic ordering of the set $\left[\binom{\substack{n}}{k}\right]=\{e_1,\ldots,e_r\}$, where $r=\binom{n}{k}$. 
	We define the following random variables:
	\[
	E_1=\min\{i:X_{e_i}=1\}.
	\]
	In other words, the variable $E_1$ denotes the random index of the first hyperedge in $H_n\in H(n,k,p)$ in lexicographic ordering. Similarly, 
	for $j=1,2,\ldots$,
	\[
	E_{j+1}=\min\{i>j:X_{e_i}=1\},
	\]
	the random index of the $(j+1)$-th hyperedge in $H_n\in H(n,k,p)$ in lexicographic ordering.
	
	We set $\mathcal{E}_0=\{\emptyset, \Omega\}$, and $\mathcal{E}_j$ denotes the sigma-algebra generated by  $E_1,\ldots,E_j$ for $j=1,2,\ldots $, that is,
	\begin{align}\label{eqn:En}
		\mathcal{E}_j=\sigma(E_1,\ldots,E_j).
	\end{align}
	Note that $\mathcal{E}_j$ reveals the information of the first $j$ hyperedges of $H_n\in H(n,k,p)$ in lexicographic ordering. Observe that, the sigma-algebra $\mathcal{E}_j$ contains the information about both the hyperedges and non-hyperedges in $H_n\in H(n,k,p)$ up to $e_{j}$.
	
	\begin{lem}\label{c-lem}Let $\lambda,k,r,n,p$ be as in Assumption \ref{assum}, and $ \{\mathcal{E}_j\}$ be as defined above. Let $H_n\in H(n,k,p)$ and $S_n=U_1(H_n)(A)$ for some fixed Borel set $A\in\mathcal{B}_m$. Then, for $i\le m$ and $N=\frac{3n\lambda}{2k}$,
		$$
		\Big|\E[S_n|\mathcal{E}_{i+1}]-\E[S_n|\mathcal{E}_i]\Big|\le \frac{2k}{n}\sum\limits_{j=1}^N\lambda^j,  \text{ ~almost surely}.
		$$
	\end{lem}
	
	Now we recall a few well known results, such as, the Doob’s Martingale Property, the \emph{Azuma-Hoeffding} inequality and the Borel-Cantelli lemma, which will be used in the proof of Proposition \ref{almost-sure-1}. Let $(\Omega, \mathcal{F}, \P)$ be a probability space, and let $\{\mathcal{F}_n\}_{n\geq 0}$ be a filtration, i.e., an increasing sequence of sigma-algebras:
	\[
	\mathcal{F}_0 \subseteq \mathcal{F}_1 \subseteq \mathcal{F}_2 \subseteq \dots \subseteq \mathcal{F}.
	\]
	A sequence of random variables $ \{X_n\} $ is called a \textit{martingale} with respect to a filtration $ \{\mathcal{F}_n\} $ if it satisfies the following conditions:
	\begin{enumerate}
		\item $ X_n $ is $ \mathcal{F}_n $-measurable for each $ n $.
		
		\item $\E[|X_n|] < \infty $ for each  $n $.
		
		\item $ \E[X_{n +1}| \mathcal{F}_{n}] = X_n$ almost surely for each  $n \geq 1 $.
	\end{enumerate}

	
	\begin{fact}{\rm(Doob’s Martingale Property \cite[Chapter VII,p.-293, Example-1]{martingle-doob})}\label{lem:doobs}
		Let $(\Omega, \mathcal{F}, \P)$ be a probability space, and let $\{\mathcal{F}_i\}_{i\geq 0}$ be a filtration.
		Suppose that $X$ is an integrable random variable, that is, $	\E[|X|] < \infty.$
		Define $M_i=\E[X | \mathcal{F}_i]$. Then the process $\{M_i \suchthat i \geq 0\}$
		is a martingale with respect to the filtration $\{\mathcal{F}_i\}$.
	\end{fact}

	\begin{fact}{\rm(  Azuma-Hoeffding Inequality \cite[Theorem 5.2]{chung2006concentration}, \cite[Theorem 1.1]{lalley2013concentration})}\label{lem:azuma}
		Let  $\{M_n\}$ be a martingale with respect to a filtration $\{\mathcal{F}_n\}$, and suppose that  there exist constants  $c_1, c_2, \dots, c_n $ such that for all  $i$,
		$$	|M_i - M_{i-1}| \leq c_i, \quad \text{almost surely}.$$
		Then, for any $t > 0$,
		$$	
		\P(|M_n - \E[M_n]| \geq t) \leq 2 \exp \left( -\frac{t^2}{2 \sum_{i=1}^{n} c_i^2} \right).
		$$
	\end{fact}
	\begin{fact}[The First Borel-Cantelli Lemma.\cite{rosenthal2006first}]\label{borel-cantelli}
		Let $(\Omega, \mathcal{F}, \P)$ be a probability space, and $\{A_n\}$ be a sequence of events such that
		$\sum_{n=1}^{\infty} {\P}(A_n) < \infty$.
		Then
		\[
		{\P}\left( \limsup_{n \to \infty} A_n \right) = 0.
		\]
	\end{fact}
	We proceed to prove Proposition \ref{almost-sure-1} assuming Lemma \ref{hyperedge bound} and Lemma \ref{c-lem}. The proofs of these lemmas will be given at the end of this section.
	
	\begin{proof}[Proof of Proposition \ref{almost-sure-1}]
		From the definition of $\mathcal E_j$, it is clear that $\mathcal{E}_j\subseteq \mathcal{E}_{j+1}$ for all $j\in \N$. In other words, the collection $\{\mathcal{E}_j\}_{j\in\mathbb{N}}$ is a filtration. Define, for $i=0,1,\ldots$,
		\[
		M_i=\E[S_n\given \mathcal E_i].
		\]
		Again, for any $A\in\mathcal{B}_m$, note that
		\begin{align*}
			S_n=U_1(H_n)(A)&=\frac{1}{n}\sum\limits_{i\in [n]}\delta_{(H_n(i),i)}(A).
		\end{align*}
		It is clear that  $\E[|S_n|]\le 1$. Therefore the Doob’s Martingale Property, Fact \ref{lem:doobs}, implies that $\{M_i\}$ is a martingale  with respect to the filtration $\{\mathcal E_i\}$.
		
		Let  $c=2k\sum\limits_{j=1}^m\lambda^i$. Then from Lemma \ref{c-lem} we get, for $i=0,1,\ldots, m$,
		\[
		|M_{i+1}-M_i|\le \frac{c}{n}, \mbox{ almost surely}.
		\]
		Therefore Fact \ref{lem:azuma} implies that, for $N=\frac{3n\lambda}{2k}$,
		\begin{align}\label{eqn:Mn}
			\P(|M_N - \E[M_N]| \geq \eps) \leq 2 \exp \left( -\frac{\eps^2n^2}{2 N c^2} \right)=2e^{-c_3n}
		\end{align}
		where $c_3=\eps^2k/3 \lambda c^2$. Note that if $\{Z_n\le N\}$ holds then $S_n$ is $\mathcal E_N$ measurable.  Therefore,  if $\{Z_n\le N\}$ occurs then $M_N=S_n$. Also observe that $\E[M_N]=\E[S_n]$. Therefore, by Lemma \ref{hyperedge bound} and \eqref{eqn:Mn}, we get
		\begin{align*}
			\P(|S_n- \E[S_n]| \geq \eps)&\le \P(\{|S_n - \E[S_n]| \geq \eps\}\cap \{Z_n\le N\})+\P(\{Z_n> N\})
			\\&\le \P(|M_N - \E[M_N]| \geq \eps) +e^{-\frac{\lambda n}{6k}}
			\\& \le 2e^{-c_3n}+e^{-\frac{\lambda n}{6k}}.
		\end{align*}
		The last inequality implies that
		$$
		\sum_{n=1}^{\infty}\P(|S_n - \E[S_n]| \geq \eps)<\infty.
		$$
		Then Borel-Cantelli Lemma,  Fact \ref{borel-cantelli}, implies that 
		$$
		\lim\limits_{n\to\infty}|S_n-\E[S_n]|=0, \mbox{ almost surely}.
		$$
		Thus, as $S_n=U_1(H_n)(A)$, we have
		$$
		\lim\limits_{n\to\infty}U_1(H_n)(A)-\E[U_1(H_n)(A)=0,  \mbox{ almost surely},
		$$
		for all $A\in\mathcal{B}_m$. The Portmanteau theorem (Theorem \ref{Portmanteau}) implies that
		$$
		\lim\limits_{n\to\infty}\left(U_1(H_n)-\E[U_1(H_n)]\right)=0,  \text{~almost surely.}
		$$
		This completes the proof.
	\end{proof}
	

	It remains to prove Lemma \ref{hyperedge bound} and Lemma \ref{c-lem}. We recall the Chernoff Bound, which will be used in the proof of Lemma \ref{hyperedge bound}.
	
	\begin{fact}[Chernoff Bound \cite{upfal2005probability}]\label{chernoff}
		Let $X_1, X_2, \dots, X_n$ be i.i.d. Bernoulli random variables with $\P(X_1= 1) = p$ and $\P(X_1= 0) = 1 - p$ for some $p\in (0,1)$. Suppose  $S_n= \sum_{i=1}^{n} X_i$. Then, for any $\delta > 0$,
		\[
		\P(S_n\geq (1+\delta)\E[S_n]) \leq e^{\left( -\frac{\delta^2 \E[S_n]}{2 + \delta} \right)}.
		\]
		
	\end{fact}    
	\begin{proof}[Proof of Lemma \ref{hyperedge bound}]
		Note that $\xi_n$ is a sum of independent Bernoulli random variables.  Then using the Chernoff Bound, Fact \ref{chernoff},
		we have 
		\[
		\P(\xi_n\ge(1+\delta)\E[\xi_n])\le e^{-\frac{\delta^2\E[
				\xi_n]}{2+\delta}}.
		\]
		Again, note that  $\E[\xi_n]=\binom{n}{k}p=n\lambda/k$.   The result follows by putting  $\delta=1/2$.
	\end{proof}

	\begin{proof}[Proof of Lemma \ref{c-lem}]
		Note that 
		\begin{align*}
			\E[S_n|\mathcal{E}_i]=\E[U_1(H_n)|\mathcal{E}_i]=\frac{1}{n}\sum\limits_{j=1}^n\P((H_n(j),j)\in A|\mathcal{E}_i).
		\end{align*}
		Therefore we have
		\begin{align*}
			&|\E[S_n|\mathcal{E}_i]-\E[S_n|\mathcal{E}_{i-1}]|\\
			=&\left|\frac{1}{n}\sum\limits_{j=1}^n\left(\P((H_n(j),j)\in A|\mathcal{E}_i)-\P((H_n(j),j)\in A|\mathcal{E}_{i-1})\right)\right|
		\end{align*}
		Let $N=\frac{3\lambda n}{2k}$. Lemma \ref{hyperedge bound} and  Fact \ref{borel-cantelli} implies that the number of hyperedges in $H_n$ is at most  $N$ almost surely, for large $n$. Thus, if the sigma-algebra $\mathcal{E}_i$ exposes the $i$-th hyperedge $e$, then the root vertex belongs to one of the $N$-radius balls centred at some vertex belonging to $e$ (see Figure \ref{fig:effected-root}).
		
		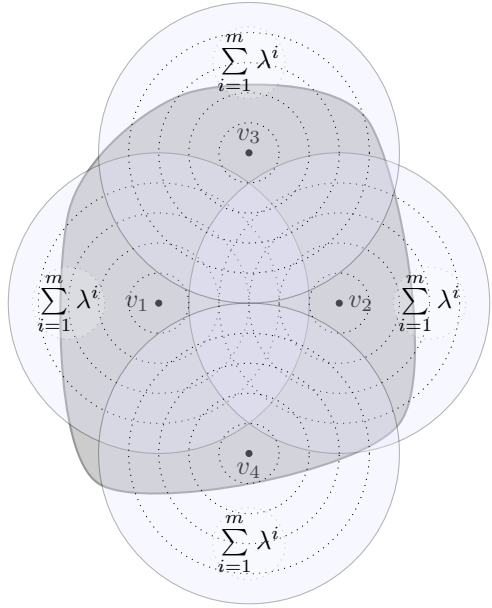
\begin{figure}[H]
			\centering
			\begin{tikzpicture}[scale=0.4]
				\coordinate (A) at (0,0);
				\coordinate (B) at (6,0);
				\coordinate (C) at (3,5);
				\coordinate (D) at (3,-5);
				
				\draw[thick,fill=gray,opacity=0.4] plot [smooth cycle] coordinates {(-2,-6) (8,-4) (7,6) (1,7) (-3,3)};

				\foreach \point in {A, B, C, D} {
					\fill[draw,scale=0.5] (\point) circle (0.2);
				}
				
				\node[left] at (A) {$v_1$};
				\node[right] at (B) {$v_2$};
				\node[above] at (C) {$v_3$};
				\node[below] at (D) {$v_4$};
				\foreach \point in {A, B, C, D} {
					\draw[fill=blue!10!white,opacity=0.3] (\point) circle (5);
					\draw[dotted] (\point) circle (1);
					\draw[dotted] (\point) circle (2);
					\draw[dotted] (\point) circle (3);
					\draw[dotted] (\point) circle (4);
				}
				\draw[dotted,fill=white,opacity=0.2](-3,0)circle (1.2);
				\node[] at (-3,0) {$\sum\limits_{i=1}^m\lambda^i$};
				\draw[dotted,fill=white,opacity=0.2](9,0)circle (1.2);
				\node[] at (9,0) {$\sum\limits_{i=1}^m\lambda^i$};
				\draw[dotted,fill=white,opacity=0.2](3,8)circle (1.2);
				\node[] at (3,8) {$\sum\limits_{i=1}^m\lambda^i$};
				\draw[dotted,fill=white,opacity=0.2](3,-8)circle (1.2);
				\node[] at (3,-8) {$\sum\limits_{i=1}^m\lambda^i$};
			\end{tikzpicture}
			\caption{For $k=4$, $m=5$, if a $4$-uniform hyperedge is exposed in a filtration, then it can affect the vertices inside the $5$-radius circles centered at $v_1,v_2,v_3$ and $v_4$.}
			\label{fig:effected-root}
		\end{figure}

		\noindent  Since $e$ contains $k$ vertices, we have total
		\[
		k\sum\limits_{j=1}^N\lambda^j
		\] 
		choices for the root being affected by revealing the $i$-th hyperedge. Since for all  $j=1,\ldots,n$, the difference  $$\frac{1}{n}\left(\P((H_n(j),j)\in A|\mathcal{E}_i)-\P((H_n(j),j)\in A|\mathcal{E}_{i-1})\right)\le \frac{2}{n}.$$ 
		Thus, \[|\E[S_n|\mathcal{E}_{i+1}]-\E[S_n|\mathcal{E}_i]|\le  \frac{2k}{n}\sum\limits_{j=1}^N\lambda^i, \text{~almost surely.}\]
		Hence the result.
	\end{proof}
	
	\begin{rem}\label{rem-cor-1}
		Recall Assumption \ref{assum}. If we relax the condition  
		$
		\binom{n-1}{k-1}p = \lambda
		$ 
		to  
		$
		\binom{n-1}{k-1}p \to \lambda,
		$ 
		the same argument remains valid. 
		In this regime, the binomial distribution $\mathrm{Bin}(\binom{n-1}{k-1}, p)$ converges 
		in law to $\mathrm{Poi}(\lambda)$, so the limiting distribution is again 
		$\Gamma_{k,\lambda}$. Furthermore, the concentration inequalities continue to hold uniformly 
		as long as $\binom{n-1}{k-1}p$ remains bounded, which also yields almost 
		sure convergence. Hence, Theorem \ref{deviationto0} is preserved under 
		the weaker assumption $\binom{n-1}{k-1}p \to \lambda$. Consequently, 
		Corollary \ref{cor-Benjamini and Schramm} follows directly from Theorem \ref{deviationto0}.
	\end{rem}

	\section{Convergence of  \texorpdfstring{$r$}{r}-set rooted uniform hypergraphs}\label{main-proof}	
	This section is devoted to the proof of Theorem \ref{main}. Like the $r=1$ case, the proof of the result is broadly divided in following two propositions.

	\begin{prop}[Convergence in expectation]\label{exp-to-nu}
		
		Let $\lambda,k,r,n,p$ be as in Assumption \ref{assum} and  $H_n\in H(n,k,p)$. 
		Then
		$$
		\E[U_r(H(n,k,p))] \rightsquigarrow\nu_{(\binom{k}{r}-1),\lambda}, \mbox{ as } n\to \infty.
		$$
	\end{prop}
	\begin{prop}[Almost sure convergence]\label{almost-sure-r}
		Let $\lambda,k,r,n,p$ be as in Assumption \ref{assum}.
		and $H_n\in H(n,k,p)$. 
		Then
		$$  \lim\limits_{n\to\infty}(U_r(H_n)-\E[U_r(H_n)])=0, \text{~almost surely}.$$
	\end{prop}
	\begin{proof}[Proof of Theorem \ref{main}]
		The result follows from Proposition \ref{exp-to-nu} and Proposition \ref{almost-sure-r}.
	\end{proof}

	We present the proofs of these propositions in the following two subsections.
	\subsection{Convergence in expectation}\label{conv-exp-r}
	Fix $1\le r\le k-1$. Note that the following proposition analogous to Proposition \ref{basis-open-convergence}.

	\begin{prop}\label{basis-open-convergence-r}
		Let $\lambda,k,r,n,p$ be as in Assumption \ref{assum} and $\mu_n^r=\E[U_r(H_n)]$. Then, for  $\delta>0$, 
		\[
		\lim\limits_{n\to\infty}\mu_n^r(B_\delta(g))=\nu_{d,\lambda}(B_\delta(g)),
		\]
		where $d=\binom{k}{r}-1$ and $\nu_{d,\lambda}$ denotes the induced measure of $d$-block Galton-Watson tree.
	\end{prop}
	
	\begin{proof}[Proof of  Proposition \ref{exp-to-nu} ]
		The result follows from Proposition \ref{basis-open-convergence-r} by following the same steps as in the proof of Proposition \ref{conv-exp-1}. We skip the details here.
	\end{proof}
	\subsubsection{Proof of $Proposition \ref{basis-open-convergence-r}$ }
	This subsection is devoted to the proof of Proposition \ref{basis-open-convergence-r}. As before, for $1\le r\le k-1$, we need to estimate, for $g\in \mathbb G^*_w$,
	\begin{align}\label{eqn:differencer}
		|\mu_n^r(B_\delta(g))-\nu_{d,\lambda}(B_\delta(g))|.
	\end{align}
	Observe that the following lemma to understand the terms $\mu_n^r(B_\delta(g))$ and $\nu_{d,\lambda}(B_\delta(g))$. If there is no confusion, for ease of writing,  we set $H\in H(n,k,p)$, $\mu_n^r=\E[U_r(H)]$ and $G_r(H)\equiv G_r(H)(S)$. Similar to Lemma \ref{EU} we have the following result.
	
	\begin{lem}\label{EUr}
		Let $\lambda,k,r,n,p$ be as in Assumption \ref{assum} and  $A\in \mathcal{B}_m$.  Suppose $S\in [\binom{n}{r}]$ be an $r$-set root. Then 
		$$
		\mu_n^r(A)=\P((G_r(H),S)\in A).
		$$
		In particular, 
		for $\delta>0$ and  $g\in\mathbb{G}^*_w$, 
		\begin{align}\label{eqn:munr}
			\mu_n^r(B_\delta(g))=\P\left((G_r(H),S)_{\lceil\frac{1}{\delta}\rceil}\cong g_{\lceil\frac{1}{\delta}\rceil}\right)
		\end{align}
		for any root $S\in [\binom{n}{r}]$.
	\end{lem}
	\begin{proof}
		The result follows by the similar arguments as in the proof of Lemma \ref{EU}.
		We skip the details. 
	\end{proof}


	Let $g$ be a rooted weighted graph and $S$ be an $r$-set. Then, 
	\begin{align}\label{eqn:nunr}
		\nu_{d,\lambda}(B_\delta(g))=\P(G_r(GW_{k-1},S)_{\lceil\frac{1}{\delta}\rceil}\cong g_{\lceil\frac{1}{\delta}\rceil}),
	\end{align}
	where $G_r(GW_{d},S)$ denotes the r-set line graph of $GW_d$ with an $r$-set root $S$.  
	Thus \eqref{eqn:munr} and \eqref{eqn:nunr} imply that the difference noted in \eqref{eqn:differencer} primarily arises from the discrepancy arises from differences between the two random $r$-set rooted $k$-uniform hypergraphs $(H_n, S)$ and $(T_k, S)$. As in the case $r = 1$ discussed in Section \ref{EXPL}, we investigate these differences through a breadth-first exploration process described below.
	
	\subsubsection{An exploration process in an $r$-set rooted hypergraphs}
	We use the breadth-first search exploration procedure as in Section \ref{EXPL}. Let $\mathcal S\subset \N_{d}^\infty$, where $d=\binom{k}{r}-1$. Define a bijection $\phi: \mathcal S \to V(G_r(H))$, where $V(G_r(H))$ denotes the vertex set of the connected component of $G_r(H)$ containing the $r$-set root $S$. Given a vertex $\jj\in V(G_r(H))$, the pre-image $\phi^{-1}(\jj)$ being defined means that $\jj$ has been explored.
	
	Let $H$ be a hypergraph and $S$ be an $r$-set. Suppose the exploration process is starting from the $r$-set root $S$ of $G_r(H)$. With abuse of notation, we set $\phi(o)=S$, $A_0^r=S$, $C_0^r=\emptyset$, $U_0^r=V(G_r(H))\setminus\{S\}$, $CE_0^r=\emptyset$, $UE_0^r=E(H)$ and proceed in the following iteration:
	
	We define $\v^r_{t+1}=\phi(\mathbf{i}_t)$, where $\mathbf{i}_t=\min\{\phi^{-1}({\v}):{\v}\in A_t\}$, the minimum is with respect to the ordering on $\mathbb{N}^\infty_{d}$. Note ${\v^r}_1=\phi(\mathbf{i}_0)=S$. The collection of \emph{unexplored hyperedges} containing $\v^r_{t+1}$ is denoted by
	\[
	R_{t+1}^r=E_{\v^r_{t+1}}(H)\cap UE_t^r, \mbox{where $E_{\v^r_{t+1}}(H)=\{e\in H\suchthat \v^r_{t+1}\subset e\}$.}
	\]
	The set of \emph{unexplored neighbours} of the vertex $\v^r_{t+1}$ is denoted by 
	\[
	J_{t+1}^r = \left\{{\v}\in V(G_r(H))\suchthat {\v}\subset \bigcup\limits_{ e \in R_{t+1}^r} e \right\} \cap U_t^r.
	\]

	\noindent Now we set: for $t\in \N\cup \{0\}$,
	\begin{enumerate}
		\item $I_{t+1}^r=\bigcup\limits_{\v\in J^r_{t+1}}\v.$
		\item {\it Active vertices:} $A_{t+1}^r=(I_{t+1}^r\cup A_t^r)\setminus{\v}_{t+1}$.
		\item {\it Unexplored vertices:} $U_{t+1}^r=U_t^r\setminus I_{t+1}^r$. 
		\item {\it Connected component:} $ C_{t+1}^r=C_t^r\cup{\v}_{t+1
		}^r$. 
		\item {\it Covered edges:} $CE_{t+1}^r=CE_{t}^r\cup R_{t+1}^r$. 
		\item {\it Unexplored edges:} $UE_{t+1}^r=UE_t\setminus R_{t+1}^r$. 
	\end{enumerate}
	The exploration ended at $t=t_0$, when $C_{t_0}^r=V(G_r(H))$. We also define the following  variables $X_0^r=|I_0^r|$, $Y_0^r=Z_0^r=|E_{{\v}_{1}}(H)|$, and for $t\in \N$,
	\begin{align}\label{eqn:randomvariables}
		X_{t}^r=|J_{t}^r|,\;\; 
		Y_{t}^r=|E_{{\v}_{t+1}}(H)|-1\;\; \mbox{ and } \;\; Z_{t}^r=|R_{t+1}^r|.
	\end{align}
	{From the above discussion, it is evident that $U^r_t\cup C^r_t\cup A^r_t=V(H)$.  Since in each step at most one $r$-set has been added to $C^r_t$, we have $|C^r_t|\le rt$. Thus,
		\begin{equation}\label{r-set-Ut-At}
			n-|U^r_t|\le rt+|A^r_t|. 
		\end{equation}
		Since we are considering $k$-uniform hypergraphs, from the definition of $J^r_t$, it is evident that $|I_t^r|\le|R^r_t|k$. Since $A_{t+1}^r=(I_{t+1}^r\cup A_t^r)\setminus{\v}_{t+1}$ and each $e\in R^r_{t+1}$ contains $\v^r_{t+1}$, it follows that $|A_{t+1}^r|-|A^r_t|\le (k-r)|R^r_{t+1}|$. Since $|A_0^r|=r$, using the difference we have 
		\begin{equation}\label{iteration-Atr}
			|A_{t+1}^r|\le r+\sum_{s=1}^{t+1}(k-r)|R^r_s|.
		\end{equation}
		Suppose further that $\mathcal{F}_t^r=\sigma(\{(A_s^r,U_s^r,C_s^r,CE_s^r,UE_s^r):s=0,1,\ldots,t\})$. 
	}
	If we replace the hypergraph $H$ with the random hypergraph $H_n\in H(n,k,p)$, then $X_t^r$, $Y_t^r$ and $Z_t^r$ become random variables. 
	{Since Assumption \ref{assum} ensures that $ \mathbf{E}(R^r_s\given \mathcal{F}_t^r)\le \lambda$, \eqref{iteration-Atr} leads us to
		\begin{equation}
			\label{exp-Ar}\mathbf{E}[A_{t}^r\given \mathcal{F}_t^r]\le r+\sum_{s=1}^{t}(k-r)\lambda=r+(k-r)\lambda t.
		\end{equation}
	}
	\begin{lem}\label{offspring-equal}
		Let $\lambda,k,r,n,p$ be as in Assumption \ref{assum} and $\{Y_t^r\}$ be as defined Equation \ref{eqn:randomvariables}. Then,
		on an enlarged probability space, there exists a sequence of i.i.d $Z_1',Z_2',\ldots ,Z_{t_0}' $ with $Z_1'\sim\text{Poisson}(\lambda)$ such that,
		for large $n$,
		$$ \P((Y_1^r,\ldots,Y_{t_0}^r)\ne (Z_1',\ldots ,Z_{t_0}'))\le r\frac{t_0(1+t_0)}{2}\frac{1}{(k-r-1)!}\frac{\lambda}{(n-r)}+\frac{\lambda t_0}{{(n-r)}^{k-r}}.$$
	\end{lem}
	\begin{proof}
		Let $Y_t^r$ and $Z_t^r$ be as defined above. Observe that, for $t\le t_0$, the following holds:
		\begin{align*}
			\P(Y_t^r=Z_t^r\given \mathcal F_t)&=\sum_{j=0}^{k-r-1}(1-p)^{\binom{n-r-|U_t|}{k-r-j}\binom{|U_t|}{j}}.
		\end{align*}
		
		\begin{figure}[t]
			\begin{tikzpicture}[scale=0.5]
				\draw[] (0,0) circle (4);
				\draw[white,fill=blue!10] (0,1) circle (3);
				\draw[] (0,2) circle (1);
				\draw[dash dot dot](0,-2.5)rectangle(2.5,0);
				\node at (0,2) {$\v^r_{t+1}$};
				\node at (-0.5,1.5) {$r$};
				\node at (1.7,-2) {$U^r_t$};
				\node at (-0.5,0.75) {$e$};
				\node at (-1.5,-1) {$k$};
				\node at (1.5,-1) {$j$};
			\end{tikzpicture}
			\caption{An $r$-set $\v^r_{t+1}$. Formation of a $k$-uniform hyperedge $e$ (blue shaded) that contains both $\v^r_{t+1}$ and $j$-elements from $U_t^r$(the region bounded by the dash-dot rectangle).}
			\label{fig:r-k-n-j}
		\end{figure}
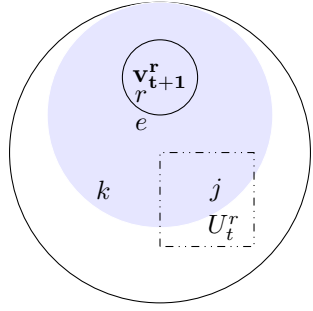
		{	A $k$-uniform hyperedge containing the $r$-set $\v^r_{t+1}$ can be formed in $\binom{n-r}{k-r}$ ways. Since $\v^r_{t+1}$ is disjoint with $U_t^r$, a $k$-uniform hyperedge containing both the $r$-set $\v^r_{t+1}$ and $U^r_t$ can be formed in $\binom{n-r-|U_t|}{k-r-j}\binom{|U_t|}{j}$ ways (see Figure \ref{fig:r-k-n-j}). Thus,
			$$\sum_{j=0}^{k-r}{\binom{n-r-|U_t|}{k-r-j}\binom{|U_t|}{j}}=\binom{n-r}{k-r}, $$}
		which implies that, by Lemma \ref{lem-inequality}, 
		\begin{align*}
			\P(Y_t^r\neq Z_t^r\given \mathcal F_t)\le \sum_{j=0}^{k-r-1}{\binom{n-r-|U_t|}{k-r-j}\binom{|U_t|}{j}}p
			&=\left[\binom{n-r}{k-r}-\binom{|U_t|}{k-r}\right]p.
		\end{align*}
		For large $n$, $\left[\binom{n-r}{k-r}-\binom{|U_t|}{k-r}\right]\approx ((n-r)^{k-r}-|U_t|^{k-r})\frac{1}{(k-r)!}$.
		Therefore, by Lemma \ref{lem-a>b ineq},
		$$  
		\P(Y_t^r\neq Z_t^r\given \mathcal F_t)\le(n-r)^{k-r-1}(n-|U_t|-r)\frac{\lambda}{(n-r)^{k-r}}=\frac{\lambda (n-|U_t|-r)}{n-r}.
		$$
		{By \eqref{r-set-Ut-At}, $n-|U_t^r|=|A_t^r|+rt$. Thus, 
			$$  
			\P(Y_t^r\neq Z_t^r\given \mathcal F_t)\le\frac{\lambda (|A_t^r|+rt-r)}{n-r}.
			$$
			
			\noindent	Again by \eqref{exp-Ar},
			$$\mathbf{E}(|A_t^r|)=\mathbf{E}(|A_t^r| \mathcal{F}_t)\le r+ \lambda (k-r)t.$$
			Therefore, $$	\P(Y_t^r\neq Z_t^r)\le\frac{\lambda(\lambda (k-r)t+rt)}{n-r}. $$
		}
		That is,
		\begin{align*}
			\P(Y_t^r\neq Z_t^r)\le \frac{c_6 t}{n-r},
		\end{align*}
		where $c_6$ is a constant (depends on $r$, $k$ and $\lambda$).
		Thus we get
		\begin{align}\label{eqn:Y_nZn}
			\P((Y_1^r,\ldots,Y_{t_0}^r)\ne (Z_1^r,\ldots , Z_{t_0}^r))&\le\sum\limits_{t=1}^{t_0}\P(Y_t^r\ne Z_t^r)\le \frac{c_6 t_0^2}{n-r}.
		\end{align}
		Since $Z_t^r\sim\text{Binomial}(\binom{n-r}{k-r},p)$, by the maximal coupling and \eqref{D_i-poi}, then there exists i.i.d. $Z_1',\ldots, Z_{t_0}'$ such that $Z_t'\sim \text{Poisson}(\lambda) $  such that 
		\[
		d_{TV}(Z_t^r,Z_t')\le\frac{\lambda}{{(n-r)}^{k-r}},
		\]
		for  large $n$. Thus, by union bound, it follows that
		\begin{align}\label{eqn:ZnZn'}
			\P((Z_1^r,\ldots,Z_{t_0}^r)\ne (Z_1',\ldots, Z_{t_0}'))\le \frac{\lambda t_0}{{(n-r)}^{k-r}}.
		\end{align}
		Therefore, \eqref{eqn:Y_nZn} and \eqref{eqn:ZnZn'} implies that 
		\[
		\P((Y_1^r,\ldots,Y_{t_0}^r)\ne (Z_1',\ldots, Z_{t_0}'))\le \frac{c_6 t_0^2}{n-r}+\frac{\lambda t_0}{{(n-r)}^{k-r}}.
		\]
		Hence the result.
	\end{proof}
	
	For an $r$-set vertex ${\v^r_{t+1}}$, we refer the edges in $R^r_{t+1}$  as \emph{offspring hyperedges} of ${\v^r_{t+1}}$. The next result ensures that any offspring hyperedge  $e\in R^r_{t+1}$ intersect with any explored $r$-set other than ${\v^r_{t+1}}$ with very low probability.
	\begin{lem}\label{treelike-1}
		Let $\lambda,k,r,n,p$ be as in Assumption \ref{assum} and $\{U_t^r\}$ be as defined above. Then
		\[
		\P(\{\exists e\in R^r_{t+1}\suchthat |e\backslash U_t^r|>r\})\le\frac{c_5}{n-r},
		\]
		for some positive constant $c_5$ (depends on $\lambda$ and $r$).
	\end{lem}
	\begin{proof} Let $A_0^r=S$ and $U_0^r=[n]\backslash A_0'$. For $t\in \N$, we denote 
		\[
		A_t^r=\cup_{{\v}\in A_t^r} {\v} \mbox{ and } U_t^r=[n]\backslash A_t^r.
		\]
		Note that the possible  number of edges containing an $r$-set $S$ is $\binom{n-r}{k-r}$. Note that if $|e\backslash U_t^r|=r$ then  $e\backslash U_t^r={\v}_{t+1}$. Therefore the number  of possible edges that satisfies $|e\backslash U_t^r|>r$ is  $\binom{n-r}{k-r}-\binom{|U_t^r|}{k-r}$. Thus, 
		$$\P(|e\backslash U_t^r|>r\given \mathcal F_t^r)\le \left(\binom{n-r}{k-r}-\binom{|U_t^r|}{k-r}\right)p.$$
		For  large $n$, by Lemma \ref{lem-a>b ineq}, note that
		\begin{align*}
			\left(\binom{n-r}{k-r}-\binom{|U_t^r|}{k-r}\right)&\approx (n-r)^{k-r}-|U_t^r|^{k-r}\le(k-r)(n-r)^{k-r-1}(n-r-|U_t^r|)).
		\end{align*}
		where $c=(k-r)\lambda$, which implies that by \eqref{exp-Ar},
		\begin{align*}
			\P(|e\backslash  U_t^r|>r)&\le \frac{c (n-r-|U_t^r|)}{{(n-r)}}\\
			&=\frac{c(n-r-\E[|U_t^r|])}{n-r}\\
			&\le \frac{c(r(t-1)-\E[|A^r_t|])}{n-r}\\
			&= \frac{c(r(t-1)-(r+(k-r)\lambda t)}{n-r}=\frac{c_5}{n-r},
		\end{align*}
		$\mbox{where }c_5\mbox{~is a constant depending on~}r,t,k,\lambda.$ Hence the result.
	\end{proof}
	
	The following lemma ensures that two distinct offspring hyperedges $e,e'\in R_{t+1}$ intersect only at ${\v^r_{t+1}}$ with a high probability. 
	\begin{lem}\label{treelike-2}
		Let $\lambda,k,r,n,p$ be as in Assumption \ref{assum} and $\{R_t^r\}$ be as defined above. Then
		\[
		\P(\{\exists e,e'\in R_{t+1}^r\suchthat e\neq e' \mbox{ and } |e\cap e'|>r\})\le \frac{c_6}{n-r}
		\]
		for some positive constant $c_6$.
	\end{lem}
	\begin{proof}
		Note that ${\v}_{t+1}\in e\cap e'$ which implies that $|e\cap e'|\ge r$. Thus $|e\cap e'|> r$ can occurs if a pair of distinct $e,e'\in E_t$ is chosen in 
		$$ \binom{n-r}{k-r}\left(\binom{n-r}{k-r}-\binom{(n-r)-(k-r)}{k-r}\right)\text{number of ways.}$$
		Therefore, by union bound, 
		$$
		\P(\{\exists e, e'\in R_{t+1}^r\suchthat |e\cap e'|>r\})\le \binom{n-r}{k-r}\left(\binom{n-r}{k-r}-\binom{(n-r)-(k-r)}{k-r}\right)p^2.
		$$
		For large $n$, observe that
		\begin{align*}
			\left(\binom{n-r}{k-r}-\binom{(n-r)-(k-r)}{k-r}\right)
			&\approx \left[(n-r)^{k-r}-(n-k)^{k-r}\right]\frac{1}{((k-r)!)^2}.
		\end{align*}
		Therefore, as $(n-r)^{k-r}-(n-k)^{k-r}\le (k-r)^2(n-r)^{k-r-1}$, we have
		\begin{align*}
			\P(\{\exists e, e'\in R_{t+1}^r\suchthat |e\cap e'|>r\})
			&\le \frac{c_6}{n-r},
		\end{align*}
		where $c_6$ is a positive constant (depends on $k,r,\lambda$).
	\end{proof}
	
	Now we are ready to give the proof of Proposition \ref{basis-open-convergence-r}.
	
	\begin{proof}[Proof of Proposition \ref{basis-open-convergence-r}]
		Let $t_0=\lceil\frac{1}{\delta}\rceil$. We consider the following events
		\begin{align*}
			E_1^r&=\{(Z'_1,\ldots,Z'_{t_0\wedge t})\ne (Y_1^r,\ldots,Y^r_{t_0\wedge t})\}.
			\\E_2^r&=\{\exists 0\le s\le t_0\wedge t\suchthat |e\backslash U_t^r|>1 \mbox{ for some } e\in R_{t}^r\}.
			\\ E_3^r&=\{\exists 0\le s\le t_0\wedge t\suchthat |e\cap e'|>1 \mbox{ for some } e\neq  e'\in R_{t}^r\}.
		\end{align*}
		Using the same arguments as in the proof of Proposition \ref{basis-open-convergence}, we have
		\[
		\mu_n^r(B_\delta(g))-\nu_{d,\lambda}(B_\delta(g))\le \P(E_1^r)+\P(E_2^r)+\P(E_3^r).
		\]
		Thus, combining Lemma \ref{offspring-equal}, Lemma \ref{treelike-1} and Lemma \ref{treelike-2} and $t_0 \wedge t \le t_0$,
		\begin{align*}
			\left|\mu_n^r(B_\delta(g))-\nu_{d,\lambda}(B_\delta(g))\right|	& \le \frac{c( t_0+\lambda t_0^2)}{(n-r)}+\frac{\lambda t_0}{(n-r)^{k-1}}
			+ \frac{c_1 t_0^2}{n-r}+\frac{c_2 t_0}{n-r}.
		\end{align*}
		Which tends to $0$ as $n\to \infty$. This completes the proof.
	\end{proof}
	
	\subsection{Almost sure convergence}\label{r-almost-sure}
	This section is dedicated to proving Proposition \ref{almost-sure-r}.
	We follow a similar approach as in the proof of Proposition \ref{almost-sure-1}. We prove the result utilizing Chernoff bounds, the Borel-Cantelli lemma, and the Azuma-Hoeffding inequality.

	\begin{lem}\label{increment-bound}
		Let $\lambda,k,r,n,p$ be as in Assumption \ref{assum}, and $ \{\mathcal{E}_j\}$ be as defined in \eqref{eqn:En}. Let $H_n\in H(n,k,p)$ and define $S_n^r=U_r(H_n)(A)$ for some fixed Borel set $A\in\mathcal{B}_{m_r}$. Then, for $i\le N$ where $N=\frac{3n\lambda}{2k}$,
		\[
		\Big|\E[S_n^r|\mathcal{E}_{i+1}]-\E[S_n^r|\mathcal{E}_i]\Big|\le \frac{2c}{\binom{n}{r}},  \text{ ~almost surely},
		\]
		where $c=\binom{c'}{r}$ and $c'=k\sum\limits_{j=1}^N\lambda^i$.
	\end{lem}
	
	\begin{proof}[Proof of Lemma \ref{increment-bound}]
		For $A\in \mathcal B_{m_r}$, we have
		\[
		S_n^r=\frac{1}{\binom{n}{r}}\sum_{S\in [\binom{n}{r}]}\delta_{(H,S)}(A).
		\]
		Define, for $i=0,1,\ldots$, 
		\[
		M_i^r=\E[S_n^r\given \mathcal E_i].
		\]
		Note that  $\E[|S_n^r|]\le 1$. Therefore the Doob's Martingale Property (Fact \ref{lem:doobs}) implies that $\{M_i^r\}$ is a martingale with respect to the filtration $\{\mathcal E_i\}$. Observe that 
		\begin{align}\label{eqn:martinglledifference}
			\left|M_{i+1}^r-M_i^r\right|=\sum\limits_{S\in \left[\substack{n\\ r}\right]}\frac{\left| \P(\{(H,S)\in A\}|\mathcal{E}_{i+1})-\P(\{(H,S)\in A\}|\mathcal{E}_i)\right|}{\binom{n}{r}}.
		\end{align}
		Note that each term in the right hand side of \eqref{eqn:martinglledifference} is bounded above by $\frac{2}{\binom{n}{r}}$. Again,
		by Lemma \ref{hyperedge bound}, the number hyperedge can be at most $N=\frac{3}{2}\frac{n\lambda}{k}$ almost surely. Thus, if the sigma-algebra $\mathcal{E}_i$ exposed the $i$-th hyperedge $e$, then the root vertex belongs to one of the $N$-radius balls centred at some vertex belonging to $e$. Since $e$ contains at most $k$ vertices, the affected $r$-set root is an $r$-set chosen from $c'=k\sum\limits_{j=1}^N\lambda^j$ vertices. Thus, the total number of choices for being an $r$-set root which is affected by revealing the $i$-th hyperedge, is $c=\binom{c'}{r}$. Therefore from \eqref{eqn:martinglledifference} it follows that 
		$$
		\left|\E[S_n^r|\mathcal{E}_{i+1}]-\E[S_n^r|\mathcal{E}_{i}]\right|\le \frac{2c}{\binom{n}{r}}, \mbox{ almost surely}.
		$$
		Hence the result follows.
	\end{proof}          
	
	\begin{proof}[Proof of Proposition \ref{almost-sure-r}]
		Lemma \ref{increment-bound} and  Fact \ref{lem:azuma} imply that
		\[
		\P(|M_N^r-\E[M_N^r]|\ge \eps)\le 2\exp\l(-\frac{\eps^2 \binom{n}{r}^2}{2Nc^2}\r)=2e^{-c_8n^{r-1}},
		\]
		where $c_8$ is a positive constant (depends on $r,k,\lambda$).
		By the same arguments as in the proof of Proposition \ref{almost-sure-1}, we have
		\[
		\sum_{n=1}\P(|S_n^r-\E[S_n^r]|\ge \eps)<\infty.
		\]
		Applying the Borel-Cantelli lemma and the Portanteau theorem, we get
		\[
		\lim_{n\to \infty}(S_n^r-\E[S_n^r])=0, \mbox{ almost surely}.
		\]
		Hence the result follows.
	\end{proof}
	\begin{rem}\label{rem-cor-2}
		In Remark \ref{rem-cor-1}, we observed that the condition $\binom{n-1}{k-1}p = \lambda$ can be replaced by $\binom{n-1}{k-1}p \to \lambda$. In the same way, replacing  
		$\binom{n-r}{k-r}p = \lambda$  
		with  
		$\binom{n-r}{k-r}p \to \lambda$ 
		leads to an identical argument.  
		Hence, Theorem \ref{main} remains valid under the
		condition $\binom{n-r}{k-r}p \to \lambda$.
		Consequently, Corollary \ref{Linial-Meshulam} follows from Theorem \ref{main}.
	\end{rem}
	
	\section{Spectral measures associated with hypergraphs}\label{spec-result} 
	We prove Theorem \ref{spe-conv} in this section. The $r$-set weighted adjacency operator 
	$A_H^{w_r} : \ell^2(\mathcal{V}_r) \to \ell^2(\mathcal{V}_r)$, defined in \eqref{eqn:A_H}, 
	coincides with the adjacency matrix of the weighted graph $G_r(H)$.
	
	It was shown in \cite[Proposition 2.2]{bordenave2016spectrum} that spectral convergence holds 
	for adjacency operators associated with unweighted graphs. That proof rely only on symmetry and boundedness assumptions, which remain true in the weighted case. To establish  Theorem \ref{spe-conv}, it 
	therefore suffices to extend that result to the setting of weighted graphs. For this extension we mainly use the fact that in the weighted case, weights only affect edge-supported functions multiplicatively, not structurally. Our strategy follows 
	the same general approach as in \cite[Proposition 2.2]{bordenave2016spectrum}, with appropriate
	modifications to account for the presence of weights. In fact, we refer directly to 
	\cite{bordenave2016spectrum} for the parts of the proof that remain unchanged 
	in the weighted setting.

	A probability measure on the space of all  rooted graphs is called unimodular when it satisfies the Mass-Transport Principle  (see \cite[Definition 2.1]{Aldous-Russell}). First, we extend this notion for the measures on the space of weighted rooted graphs $\mathbb{G}_w^*$.
	A two-rooted weighted graphs $(G,w,o,o')$ is a weighted graphs with weight $w$ and two roots (distinguished vertices) $o$ and $o'$. Two graphs $(G_1,w_1,o_1,o_1')$ and $(G_2,w_2,o_2,o_2')$ are called isomorphic if there exists a bijective map $f:V(G_1)\to V(G_2)$ such that $(G_1,w_1,o_1)\simeq (G_2,w_2,o_2)$ with respect to $f$ and $f(o_1')=o_2'$. 
	
	Let $\mathbb{G}^{**}_w$ denote the space of all equivalence classes of two-rooted locally finite weighted graphs.
	A Borel probability measure $\mu$ on $\mathbb{G}_w^*$ is called {\it unimodular} if it	satisfies the Mass-Transport Principle, that is, for each Borel measurable function $f:\mathbb{G}_w^{**}\to[0,\infty)$ the following holds:
	\begin{align}\label{eqn:MTP}
		\E_{\mu}\left[ \sum_{v \in V(G)} f(G(o),w, o, v)\right] =\E_{\mu}\left[ \sum_{v \in V(G)} f(G(o),w, v, o)\right],
	\end{align}
	where $G(o)$ denotes the connected component containing the vertex $o$ and 
	$$ \E_{\mu}\left[ \sum_{v \in V(G)} f(G(o), w,o, v)\right]=\int_{\mathcal{G}_*} \sum_{v \in V(G)} f(G(o),w, o, v) \, d\mu(G,w, o).$$
	The space of all unimodular measures on $\mathbb G_w^*$ is denoted by $\mathcal P_{uni}(\mathbb G_w^*)$. Note that the following reduction, analogous to \cite[Proposition 2.2]{Aldous-Russell}
	\begin{lem}\label{lem:AL}
		Let $\mu\in \mathcal P(G_w^*)$. Then $\mu$ is unimodular if and only if \eqref{eqn:MTP}
		holds for all measurable functions $f:\mathbb{G}_w^{**}\to[0,\infty)$ such that $f(G,w,u,v)=0$ if $\{u,v\}\notin E(G)$. 
	\end{lem}
	\begin{proof}
		The proof is same as the proof of Proposition 2.2 in \cite{Aldous-Russell}. We skip the details.
	\end{proof}
	
	The following result shows that the uniform measure on every finite weighted graphs is unimodular.
	\begin{lem}\label{lem:unimodular-graph}
		Let  $(G,w)$ be a  weighted graph on $n$ vertices and $U(G,w)$ be as defined in \eqref{UG-graph}. Then  the measure $U(G,w)$ is unimodular.
	\end{lem}

	\begin{proof}[Proof of Lemma \ref{lem:unimodular-graph}]
		Let $f:\mathbb{G}_w^{**}\to[0,\infty)$ be a measurable function  such that $f(G,w,u,v)=0$ if $\{u,v\}\notin E(G)$. Observe that, if ${u,v} \in E(G)$ then $(G(u),w, u,v)= (G(v),w,u,v)$. It follows that
		\begin{align*}
			\E_{U(G,w)}\l[ \sum_{v \in V} f(G(o), w,o, v)\r]
			&= \frac{1}{n} \sum_{u \in V} \sum_{v \in V(G(u))} f(G(u),w, u, v)\\
			&= \frac{1}{n} \sum_{v \in V} \sum_{u \in V(G(v))} f(G(v), w,u, v)\\
			&= \E_{U(G,w)} \l[\sum_{u \in V} f(G(o),w, u, o)\r].
		\end{align*}
		This completes the proof of the result by Lemma \ref{lem:AL}.
	\end{proof}

	\begin{lem}\label{lem:unimodular}
		Let $H_n\in H(n,k,p)$ and $U_r(H_n)$ be the measure as defined in \eqref{u_r(H)}. Then $U_r(H_n)$ is unimodular. 
	\end{lem}
	\begin{proof}
		Since $U_r(H_n)=U(G_r(H_n))$, where $G_r(H)$ is an weighted graph, thus the lemma follows directly from Lemma \ref{lem:unimodular-graph}.
	\end{proof}
	Recall the von Neumann algebra $\mathcal{M}$ of graph-operators.
	In the next lemma we show that unimodularity of $\rho$ ensures that $\Tr_{\rho}(AB)=\Tr_{\rho}(BA)$ for all $A,B\in\mathcal{M}$, that is the operator $\Tr_{\rho}$ is tracial. The proof of the result is motivated by a discussion in \cite[Section-5]{Aldous-Russell}.
	\begin{lem}
		For $\rho\in \mathcal{P}_{uni}(\mathbb{G}^*_w)$, we have $\Tr_\rho(AB)=\Tr_{\rho}(BA)$ for all $A,B\in\mathcal{M}$.
	\end{lem}
	\begin{proof}
		Since $\rho\in \mathcal{P}_{uni}(\mathbb{G}^*_w)$, it satisfies the Mass-Transport
		Principle
		$$\int \sum_{v \in V} f(G,o,v) \, d\rho(G,o) 
		\;=\; 
		\int \sum_{v \in V} f(G,v,o) \, d\rho(G,o).$$
		That is, $\E_\rho \!\left[ \sum_{v \in V} f(G,o,v) \right]
		\;=\;
		\E_\rho \!\left[ \sum_{v \in V} f(G,v,o) \right].$ 
		Thus, for all $A,B\in\mathcal{M}$
		\begin{align}\label{trace-cal}
			\Tr_{\rho}(AB)&=\E_\rho\langle AB\mathbf{1}_o,\mathbf{1}_{o}\rangle=\E_\rho\langle B\mathbf{1}_o,A^*\mathbf{1}_{o}\rangle
			\\&=\E_\rho[\sum\limits_{v\in V(G)}\langle B\mathbf{1}_o,\mathbf{1}_v\rangle\langle \mathbf{1}_v,  A^*\mathbf{1}_{o}\rangle]\nonumber
			\\\notag
			&=\E_\rho[\sum\limits_{v\in V(G)}\langle B\mathbf{1}_o,\mathbf{1}_v\rangle\langle A\mathbf{1}_v,  \mathbf{1}_{o}\rangle].
		\end{align}
		Assume that $f(G,o,v)=\langle B\mathbf{1}_o,\mathbf{1}_v\rangle\langle A\mathbf{1}_v,  \mathbf{1}_{o}\rangle$. Then MTP implies that 
		\begin{align}\label{trace-uni}
			\E_\rho[\sum\limits_{v\in V(G)}\langle B\mathbf{1}_o,\mathbf{1}_v\rangle\langle A\mathbf{1}_v,  \mathbf{1}_{o}\rangle]
			&=\E_\rho[\sum\limits_{v\in V(G)}\langle B\mathbf{1}_v,\mathbf{1}_o\rangle\langle A\mathbf{1}_o,  \mathbf{1}_{v}\rangle]\\\notag
			&=\E_\rho[\sum\limits_{v\in V(G)}\langle A\mathbf{1}_o,  \mathbf{1}_{v}\rangle\langle \mathbf{1}_v,B^*\mathbf{1}_o\rangle]
			\\&=\Tr_{\rho}(BA).\nonumber 
		\end{align}
		Therefore, \eqref{trace-cal} and \eqref{trace-uni} leads us to $\Tr_\rho(AB)=\Tr_\rho(BA)$.
	\end{proof}
	Let us recall the following result on rank from \cite[Theorem-3]{nelson}.
	
	\begin{lem}[\cite{nelson}]\label{nelson}
		Let $A$ and $B$ be closed operators affiliated with $\mathcal{M}$. 
		Suppose that for all $\varepsilon > 0$ there exists a projection 
		$P \in \mathcal{M}$ with $rank(I-P) < \varepsilon$ such that whenever 
		$y \in \mathrm{Dom}(A)$ with both $y$ and $Ay$ in the range of $P$, 
		then $y \in \mathrm{Dom}(B)$ and $By = Ay$, and conversely. 
		(In particular, this is the case if $PH \subset \mathrm{Dom}(A) \cap \mathrm{Dom}(B)$ 
		and $AP = BP$.) Then $A = B$.
	\end{lem}
	A densely defined operator $A$ is called \emph{self-adjoint} if 
	$A = A^{*}$, meaning that $A$ is symmetric and 
	$\operatorname{Dom} A = \operatorname{Dom} A^{*}$. Equivalently, 
	a closed symmetric operator $A$ is self-adjoint precisely when 
	$A^{*}$ is symmetric. For a self-adjoint operator $A$, the quantity 
	$\langle x, Ax \rangle$ is real for every $x \in \operatorname{Dom} A$; 
	that is,
	$$
	\langle x, Ax \rangle = \overline{\langle Ax, x \rangle} 
	= \overline{\langle x, Ax \rangle} \in \mathbb{R}, 
	\quad \mbox{ for all }\, x \in \operatorname{Dom} A.
	$$
	A symmetric operator $A$ is \emph{essentially self-adjoint} if its closure 
	is self-adjoint.
	
	Now, we are in the position to state and proof the following result which is motivated by \cite[ Proposition 2.2, Statement (i)]{bordenave2016spectrum}. 
	
	\begin{lem}\label{lem:essentiallyselfadjoint}
		Let $A^w$ be a weighted adjacency operator of a weighted locally finite graph as defined in \eqref{eqn:A_H} and $\rho\in \mathcal P_{uni}(\mathbb G^*_w)$. Then $A^w_G$ is  an essentially self adjoint operator $\rho$-a.s.
	\end{lem}
	\begin{proof}
		For any weighted graph $(G,w)$, consider the weighted adjacency operator $A^w_G$. Proceeding like \cite[ Proposition 2.2, Statement (i)]{bordenave2016spectrum}, we can conclude the operator $A^w:(G,o)\mapsto A^w_G$ the closure $\bar{A^w}$, and the adjoint $(A^w)^*$ are affiliated with $\mathcal{M}$. Let
		$$
		V_n(G)=\{v\in V(G):\deg_G(v)\le n\},
		$$
		be the set of vertices in $G$ whose degree does not exceed $n$. Moreover, whenever $\{u,v\}\in E(G)$, we require $\deg_G(u)\le n$ and $\deg_G(v)\le n$.
		
		Let $P_n$ denote the projection onto the subspace  
		$$
		{\mathcal H}_n=\{\mathbf{f}\in \mathcal{H}:\rho\text{-a.s. }\operatorname{supp}(\mathbf{f}(G,o))\subseteq V_n(G)\},
		$$  
		For $\mathbf{f}\in \mathcal{H}_n$ and $v\in V(G)$, we have
		\[
		(A^w_G \mathbf{f})(v) = \sum_{u\sim v} w(u,v)\,\mathbf{f}(u).
		\]
		By the Cauchy--Schwarz inequality,
		\[
		\big|(A^w_G \mathbf{f})(v)\big|^2
		\;\le\;
		\left(\sum_{u\sim v} |w(u,v)|^2\right)
		\left(\sum_{u\sim v} |\mathbf{f}(u)|^2\right)
		\;\le\;
		\left(\sum_{u\sim v} |w(u,v)|^2\right)\|\mathbf{f}\|^2.
		\]
		
		Thus, $\rho$-a.s.
		\begin{align*}
			\|(A^w_G P_n)\mathbf{f}\|^2
			&= \sum_{v\in V(G)} \mathbf{1}_{\{\deg(v)\le n\}}
			\,\big|(A^w_G \mathbf{f})(v)\big|^2 \\
			&\le \sum_{v\in V(G)} \mathbf{1}_{\{\deg(v)\le n\}}
			\left(\sum_{u\sim v} |w(u,v)|^2\right)\|\mathbf{f}\|^2.
		\end{align*}
		
		If the weights are uniformly bounded, i.e.
		\[
		M := \sup_{u,v} |w(u,v)| < \infty,
		\]
		then
		\[
		\sum_{u\sim v} |w(u,v)|^2 \le M^2 \deg(v) \le M^2 n,
		\]
		and therefore
		\[
		\|(A^w_G P_n)\mathbf{f}\|^2 \;\le\; (M^2 n^2)\,\|\mathbf{f}\|^2.
		\]
		
		Hence, $A^w_G P_n$ is a bounded operator with operator norm at most $Mn$.

		Because $A_G^w P_n$ is bounded, it is everywhere defined and extends $A^w$ on $\mathcal H$ (since $P_n f = f$ for $f \in \mathcal H_n$).
		Thus, for any $h \in \mathcal H_n$, note that 
		$
		A^w h = A^w P_n h,
		$
		which is well-defined and lies in $\mathcal H$. Hence 
		$
		\mathcal H_n \subset \operatorname{Dom}(\overline{A^w}).
		$

		Similarly, boundedness of $A^w P_n$ implies that its adjoint $(A^w P_n)^*$ is bounded, and one has 
		$
		(A^w P_n)^* = P_n (A^w)^*
		$
		on the appropriate domain. Therefore,
		$
		\mathcal H_n \subset \operatorname{Dom}((A^w)^*). 
		$
		For $\mathbf{f},\mathbf{g}\in\mathcal{H}_n$, Since the weight is symmetricthat is $w(u,v)=w(v,u)$ for any two vertices $u,v$, 
		$$\langle A^w_G \mathbf{f}_{(G,o)},\mathbf{g}_{(G,o)}\rangle =\langle  \mathbf{f}_{(G,o)},A^w_G \mathbf{g}_{(G,o)}\rangle.$$
		Since $P_n f = f$ for $f \in \mathcal H_n$, it follows that $A^w$ and $(A^w)^*$ coincide on $\mathcal H_n$. In addition, because $\rho$ is a probability measure on locally finite graphs, we obtain  
		$$
		\rho\Big(\{ \deg(o) > n \;\;\text{or}\;\; \exists v : (v,o) \in E \ \text{with}\ \deg(v) > n \}\Big)=\epsilon(n) \xrightarrow[n \to \infty]{} 0.
		$$
		Since for any $ B \in \mathcal{M} $,
		$$
		\operatorname{nullity}(B) = \mu_B(\{0\}) = \operatorname{Tr}_\rho\big(E^B(\{0\})\big)
		$$
		and hence
		$$
		\operatorname{rank}(B) = 1 - \operatorname{Tr}_\rho\big(E^B(\{0\})\big),
		$$
		we deduce from \eqref{rank-null} that
		$$
		\operatorname{rank}(\mathcal H_n) \, \P\big(e_o \in V_n(G)\big) \longrightarrow 1 \quad \mbox{ as } n \to \infty.
		$$
		Since both $A^w$ and $(A^w)^*$  is affiliated with $\mathcal{M}$, by Lemma \ref{nelson}, we have $\bar{A^w}=(A^w)^*$
		
	\end{proof}

	Let $A$ be a self-adjoint operator on a Hilbert space $\mathcal H$. The {\it resolvent set} of $A$ is denoted by $\rho(A)$, that is, $\rho(A)=\{\lambda\in \C\suchthat (\lambda I-A) \mbox{ is invertible}\}$. Note that  $\rho(A) \supset \mathbb{C} \setminus \mathbb{R}$ as $A$ is self-adjoint. The {\it resolvent} of $A$ is denoted by $R(A,\lambda)$, that is,
	\[
	R(A,\lambda):=(\lambda I-A)^{-1}, \mbox{ for } \lambda\in \rho(A).
	\]
	For each $\lambda$ in the resolvent set $\rho(A)$, the operator $R(A,\lambda)$ exists and is bounded. In particular, $R(A, \lambda)$ is well defined on $\C\backslash \R$.

	Let $A_n$ and $A$ be self-adjoint operators on a Hilbert space $\mathcal{H}$. We say that $A_n$ converges to $A$ in the \textit{strong resolvent sense} if, for some (and therefore every) $\lambda \in \mathbb{C} \setminus \mathbb{R}$, the resolvent operators satisfy
	$$
	R(A_n,\lambda) \psi \to R(A_n,\lambda) \psi \quad \text{for all } \psi \in \mathcal{H}.
	$$
	In other words, $R(A_n,\lambda)$ converges to $R(A,\lambda)$ under the strong topology.

	\begin{lem}\label{strng resolvent to weak con}
		Let $A_n$ be a sequence of essential self-adjoint operators on a Hilbert space $\mathcal{H}$ that converges to an essential self-adjoint operator $A$ in the strong resolvent sense.
		For a fixed unit vector $\psi \in \mathcal{H}$, let $\mu_{A_n}^\psi$ and $\mu_A^\psi$ denote the scalar spectral measures associated with $A_n$ and $A$, respectively as defined in \eqref{spectral-vector-measure}.  
		Then
		$$
		\mu_{A_n}^\psi \rightsquigarrow \mu_A^\psi, \mbox{ as } n\to \infty.
		$$
		That is, $\mu_{A_n}^\psi$ converges weakly to $\mu_A^\psi$.
	\end{lem}
	\begin{proof}
		
		Since $A_n$ converges to $A$ in the strong resolvent sense, it follows that $f(A_n)\psi$ converges in norm to $f(A)\psi$ \cite[Theorem VIII.20]{FA-reed-barry}. Consequently, by continuity, 
		$$
		\langle \psi, f(A_n)\psi \rangle \to \langle \psi, f(A)\psi \rangle.
		$$
		The spectral theorem, for every bounded continuous function $f: \mathbb{R} \to \mathbb{C}$,  implies
		$$
		\langle \psi, f(A_n)\psi \rangle = \int_{\mathbb{R}} f(\lambda) \, d\mu_{A_n}^\psi(\lambda) \mbox{ and }
		\langle \psi, f(A)\psi \rangle = \int_{\mathbb{R}} f(\lambda) \, d\mu_A^\psi(\lambda).
		$$
		Thus, for all bounded continuous functions $f$, observe that 
		$$
		\int f \, d\mu_{A_n}^\psi \to \int f \, d\mu_A^\psi.
		$$
		Which implies that  $\mu_{A_n}^\psi$ converges to $\mu_A^\psi$ weakly.
	\end{proof}
	
	\begin{lem}\label{lem:congenceofstieltjes}
		Let $\rho_n,\rho\in \mathcal P_{uni}(\mathbb G^*_w)$. Suppose $\rho_n\rightsquigarrow \rho$ as $n\to \infty$ then
		\[
		\mu_{\rho_n}\rightsquigarrow \mu_\rho, \mbox{ as} n\to \infty.
		\]
	\end{lem}
	\begin{proof}
		Let $(\rho_n)$ be a sequence converging to $\rho$ in the local weak topology. By Skorokhod’s  representation theorem \cite[Appendix, Theorem 7]{bordenave2010resolvent}, one may construct a common probability space on which the rooted graphs $(G_n,o)$ converge locally to $(G,o)$, where $(G_n,o)$ has law $\rho_n$ and $(G,o)$ has law $\rho$.
		
		Observe that if $\P$ denotes the probability measure governing the joint distribution of $(G_n,o)$ and $(G,o)$, then by Lemma \ref{lem:essentiallyselfadjoint}, $\P$-a.s.\ both $A_n$ and $A$ are essentially self-adjoint with common core given by the compactly supported functions in $\ell^2(V)$. On the other hand, for any compactly supported $f \in \ell^2(V)$, one has $A_n f = A f$ for all large $n$, where $A_n$ and $A$ denote the adjacency operators of $G_n$ and $G$, respectively.
		Therefore $A_n$ converges to $A$ in  \emph{strong resolvent convergence} (see \cite[Theorem VIII.25
		]{FA-reed-barry}).

		Then Lemma \ref{strng resolvent to weak con} implies that a.s. $\mu_n^{\mathbf{1}_o}$ converges weakly to $\mu^{\mathbf{1}_o}$. 
		
		Taking expectation, we get $\mu_{\rho_n} = \E \mu_n^{\mathbf{1}_o}$ converges weakly to $\mu_\rho = \E \mu^{\mathbf{1}_o}$. 
		In fact,  The random variables \( X_n := \int f \, d\mu^{1_o}_{n} \) are all bounded (since \( f \) is bounded and \( \mu^{1_o}_{n} \) are probability measures). Then, one can apply the \emph{Bounded Convergence Theorem (BCT)}:
		\[
		\E[X_n] \to \E[X],
		\quad \text{where} \quad
		X = \int f \, d\mu^{\mathbf{1}_o}.
		\]
		Therefore, $\mu_{\rho_n}\rightsquigarrow\mu_\rho$ as $n\to \infty$.
	\end{proof}
	
	Finally we give the proof of  Theorem \ref{spe-conv}. 
	\begin{proof}[Proof of  Theorem \ref{spe-conv}]
		If $H_n\in H(n,k,p)$ and $\rho_n=U(G_r(H_n))$, and $A=A^{w}$, the weighted adjacency matrix, then For any Borel set $B$, 
		\begin{align*}
			\mu_{\rho_n}(B)=\frac{1}{n}\sum_{i=1}^n\int_B\langle dE^A(x)\mathbf{1}_i,\mathbf{1}_i\rangle=\int_B\frac{1}{n}\sum_{i=1}^n\langle dE^A(x)\mathbf{1}_i,\mathbf{1}_i\rangle
		\end{align*}
		Since $A$ has orthonormal eigen basis $\{v_1,\ldots,v_n\}$, and the trace operator is invariant under orthonormal transformation, therefore,
		\begin{align*}
			\int_B\frac{1}{n}\sum_{i=1}^n\langle dE^A(x)\mathbf{1}_i,\mathbf{1}_i\rangle&=\int_B\frac{1}{n}\sum_{i=1}^n\langle dE^A(x)v_i,v_i\rangle
			\\&=\frac{1}{n}\sum_{i=1}^n\delta_{\lambda_i}(B)\langle v_i,v_i\rangle
			=\frac{1}{n}\sum_{i=1}^n\delta_{\lambda_i}(B).
		\end{align*} 
		That is, $\mu_{\rho_n}=\mu_{A}$, the spectral measure of $A$.
		Since the $r$-set adjacency matrix defined in \eqref{eqn:A_H} is an weighted adjacency matrix of the graph $G_r(H)$, Theorem \ref{spe-conv} directly follows from Lemma \ref{lem:congenceofstieltjes}. This completes the proof.
	\end{proof}
	
	\begin{rem}
		Using arguments analogous to those in Remark \ref{rem-cor-1} and Remark \ref{rem-cor-2}, we see that Corollary \ref{spec-ER} and Corollary \ref{spec-r=k-1} remain valid under the condition  
		$$\binom{n-r}{k-r}p \to \lambda.$$
	\end{rem}

\end{document}